\documentclass[english]{article}

    \newcommand{\mytitle}
    {Symmetric Systems and their Applications to Root         
      Systems Extended by Abelian Groups}                     

    \usepackage[all]{xy}
    \usepackage{amssymb,amsmath,latexsym}
    \usepackage{eucal}
     \usepackage[
       pdftitle={\mytitle},
       ]{hyperref}
    \usepackage{theorem}

    \newtheorem{lem}{Lemma}[section]
    \newtheorem{thm}[lem]{Theorem}
    \newtheorem{prop}[lem]{Proposition}
    \newtheorem{cor}[lem]{Corollary}
    {\theorembodyfont{\rm}
      \newtheorem{defn*}[lem]{Definition}
      \newtheorem{rem*}[lem]{Remark}
      \newtheorem{exmp}[lem]{Example}

      \newtheorem{constr*}[lem]{Construction}
    }
    \newenvironment{constr}{\begin{constr*}}{\hspace*{\fill}
        \raisebox{.0mm}{$\diamond$}\end{constr*}}
    \newenvironment{rem}{\begin{rem*}}{\hspace*{\fill}
        \raisebox{.0mm}{$\diamond$}\end{rem*}}
    \newenvironment{defn}{\begin{defn*}}{\hspace*{\fill}
        \raisebox{0mm}{$\diamond$}
      \end{defn*}}
    \newenvironment{lastdisplay}%
    {\par\noindent~\begin{minipage}[b]{0.97\linewidth}}%
    {\end{minipage}}
    \newenvironment{lastalign*}%
    {\par\noindent\begin{minipage}[b]{0.95\linewidth}
        \medskip\begin{center}\begin{math}\displaystyle}%
    {\end{math}\end{center}
          \end{minipage}}
    \newenvironment{lastenumerate}%
    {\par\medskip\noindent\begin{minipage}[b]{0.95\linewidth}
      \begin{enumerate}}%
    {\end{enumerate}\end{minipage}}
    \newenvironment{proof}{\par\medskip\noindent\upshape\rmfamily
      \textbf{Proof}.\ }{\hspace*{\fill}\rule{1.2ex}{1.2ex}\par\medskip}
    \newenvironment{proof*}{\par\medskip\noindent\upshape\rmfamily
      \textbf{Proof}.\ }{\par\medskip}
    \newenvironment{lastequation*}{
      \par\medskip\noindent\hspace*{\fill}\begin{math}\displaystyle
      }{
      \end{math}\hspace*{\fill}\rule{1.2ex}{1.2ex}\par}
    \newenvironment{lasteqnarray*}{
      \par
      \begin{minipage}[t]{\hfill}
        \begin{eqnarray*}
        }{      
        \end{eqnarray*}
      \end{minipage}
      \begin{minipage}[b]{0.5cm}
        \rule{1.2ex}{1.2ex}
      \end{minipage}
    }
    
    \renewcommand{\emptyset}{\varnothing}
    \newcommand{\newop}[2]{\newcommand{#1}{\mathop{\mathsf{\strut #2}}\nolimits}}
    \newop{\Hom}{\mathrm{Hom}}
    \newop{\Aut}{\mathrm{Aut}}
    \newop{\End}{\mathrm{End}}
    \newop{\co}{\mathrm{co}}
    \newop{\id}{\mathrm{id}}
    \newop{\im}{\mathrm{im}}
    \newop{\rk}{\mathrm{rk}}
    \newop{\spann}{\mathrm{span}}
    \newop{\conv}{\mathrm{conv}}
    \newop{\rank}{\mathrm{rank}}
    \newop{\GL}{\mathrm{GL}}
    \newop{\PDS}{\mathrm{PDS}}
    \newop{\Oo}{\mathrm{O}}
    \newop{\cone}{\mathrm{cone}}
    \newop{\algint}{\mathrm{algint}}
    \newop{\inn}{\mathrm{int}}
    \newop{\aff}{\mathrm{aff}}

    \newcommand{\Q}{\mathbb{Q}}

    \newcommand{\Z}{\mathbb{Z}}
    \newop{\M}{\mathbf{M}}
    \newcommand{\eff}{\mathrm{eff}}
    \newcommand{\fix}{\mathrm{fix}}
    \newop{\PGL}{\mathrm{PGL}}
    \newop{\sL}{\mathrm{sl}}
    \newop{\tr}{\mathrm{tr}}
    \newcommand{\sbullet}{\hspace{0.5mm}\begin{picture}(1,5)(-.5,-3)\circle*{2}\end{picture}\hspace{0.5mm}}

    \newcommand{\propname}[1]{{\rm\bf(#1)}}
    \renewcommand{\phi}{\varphi}
    \renewcommand{\check}[1]{#1^\vee}

    \newcommand{\red}{\mathrm{red}}

    \newcommand{\ab}{\mathrm{ab}}
    
    \newcommand{\sh}{\mathrm{sh}}
    \newcommand{\lng}{\mathrm{lg}}
    \newcommand{\ex}{\mathrm{ex}}
    \newcommand{\rd}{\mathrm{red}}

    \newcommand{\trim}{\mathrm{trim}}
    \newcommand{\calU}{\mathcal{U}}
    \newcommand{\calN}{\mathcal{N}}
    \newcommand{\calV}{\mathcal{V}}
    \newcommand{\calP}{\mathcal{P}}
    \newcommand{\calQ}{\mathcal{Q}}
    \newcommand{\calW}{\mathcal{W}}
    
    \newcommand{\calA}{\mathcal{A}}
    \newcommand{\calB}{\mathcal{B}}
    \newcommand{\calK}{\mathcal{K}}
    \newcommand{\calX}{\mathcal{X}}
    
    \newcommand{\calY}{\mathcal{Y}}
    \newcommand{\calZ}{\mathcal{Z}}
    \newcommand{\calT}{\mathcal{T}}
    
    \newcommand{\calL}{\mathcal{L}}

\begin{document}

    \title{\mytitle\footnote{ 
    The research for this article was supported by a postdoctoral
    fellowship at 
    the Department of Mathematics and Statistics at
    the Dalhousie University.
    }}
    \author{Georg W. Hofmann}
    
    \maketitle
    \begin{abstract}
      We investigate the class of root systems $R$ 
      obtained by extending an
      irreducible root system by a torsion-free group $G$. In this
      context there is a Weyl group $\calW$ and a group $\calU$ with
      the presentation by conjugation. We show under additional
      hypotheses that the kernel of the
      natural homomorphism $\calU\to\calW$ is isomorphic to the kernel
      of $\calU^\ab\to\calW^\ab$, where $\calU^\ab$ and $\calW^\ab$ 
      denote the
      abelianizations of $\calU$ and $\calW$ respectively. For this we
      introduce the concept of a symmetric system, a discrete version
      of the concept of a symmetric space.
      \hfill\break
      \hfill\break                      
    \end{abstract} 
    
    \pagestyle{headings}
    \tableofcontents

    \section{Introduction}                                           %
    Extended affine root systems (EARSs) 
    belong to the theory of
    extended affine Lie algebras (EALAs). The class of EALAs
    is a generalization of the class of affine Kac-Moody
    algebras. In \cite{EALA_AMS} EALAs are introduced axiomatically and
    the EARSs are classified.
    Lie algebras in this class were studied by physicists under the
    name of irreducible quasi-simple Lie algebras in
    \cite{torresani}. 
    Several years earlier, a certain subclass of EALAs was already studied in
    \cite{saito1}. In \cite{yoshii-rootabelian} the notion of
    a root system extended by an abelian group $G$ is
    introduced. This more modern approach generalizes 
    the notions of affine root systems given in \cite{EALA_AMS} 
    and \cite{saito1} and does not rely on the idea of the root
    system being embedded discretely in a real vector
    space. Arguments using discreteness are replaced by algebraic
    ones and the abelian group $G$ is no longer necessarily free and
    finitely generated.
    
    In the context of EARSs a group $\calU$ is given by the so-called
    presentation by conjugation:
    \begin{align*}
      \calU
      \cong
      \big\langle(\hat r_\alpha)_{\alpha\in R^\times}
      ~|~
      &
      \hat r_\alpha=\hat r_\beta
      \text{~~~if $\alpha$ and $\beta$ are linearly dependent,}\\
      &\hat r_\alpha^2=1,~~
      \hat r_\alpha\hat r_\beta\hat r_\alpha^{-1}
      =\hat r_{r_\alpha(\beta)}; 
      \text{ for }\alpha,\beta\in R^\times\big\rangle,
    \end{align*} 
    where $R^\times$ stands for the set of anisotropic roots.
    There is a natural group 
    homomorphism from $\calU$ into the Weyl group $\calW$. 
    Several results are known about whether this homomorphism is
    injective. 
    (See for example
    \cite{krylyuk},
    \cite{EAWG},
    \cite{azamReduced},
    \cite{MR2341017},
    \cite{azamA1},
    \cite{myWeyl}.) 
    
    This article attempts to modernize the theory of presentation
    by conjugation of Weyl groups by
    extending the class of root systems under
    consideration to the class proposed by \cite{yoshii-rootabelian}
    and by choosing a systematic approach using the new notion of a
    symmetric system. 
    This is a discrete version of the notion 
    of a symmetric space treated in \cite{MR0239005}. With our approach
    the group $\calU$ is the initial element of a certain
    category of reflection groups. 
    We generalize the notion of the Weyl group $\calW$ to the case
    where $G$ is torsion-free and $R$ is tame, that is, $G$ 
    possesses a so-called twist
    decomposition if $R$ is of non-simply laced type. 
    (This decomposition is a
    generalization of a decomposition that exists, if $G$ is finitely
    generated and free abelian.) 
    We obtain the following characterization of the
    relationship between the group $\calU$ and the Weyl
    group~$\calW$. 

    Suppose $R$ is a tame irreducible reduced
    root system extended by a free abelian group
    $G$.
    Denote the abelianizations of $\calU$
    and $\calW$ by $\calU^\ab$ and $\calW^\ab$, respectively.
    \par\medskip\noindent
    {\bf Theorem}~
     The natural homomorphism 
     \begin{math}
       {\phi:
       \ker(\calU\to\calV)
       \to
       \ker(\calU^\ab\to\calW^\ab)}
     \end{math}
     is an isomorphism.
    \par\medskip\noindent
    This is the main result of this article and it is proved in the
    last section. We explore the notion of a symmetric system in
    Sections 2 to 5 only as far as needed to prove this result. In
    Section~6 we obtain the necessary results about finite irreducible
    root systems and the action of their Weyl group on the root
    lattice and the coroot lattice.

    Among other corollaries we derive the following
    result from the above theorem by investigating the orbits of
    $\calW$ in the
    root system:
    Suppose $R$ is a
    tame irreducible reduced root system extended by a free abelian group
    $G$. 
    \par\medskip\noindent
    {\bf Corollary}~
    If $R$ is not of type $A_1$,
    $B_\ell(\ell\ge2)$ and 
    $C_\ell(\ell\ge3)$, then
    $\calU\to\calW$ is injective.
    \par\medskip\noindent

    In other words, in this case, the Weyl group has the presentation
    by conjugation. This is a generalization of results known in the
    case where $G$ is a finitely generated free abelian group. We
    expect that various other results will be generalized 
    using the theorem above and by investigating the groups
    $\calU^\ab$ and $\calW^\ab$, which are elementary abelian
    two-groups and thus are accessible to arguments of linear algebra over
    the Galois field $\mathrm{GF}(2)$.

    \section{Symmetric systems}                                      %

    In this section we introduce the notion of symmetric system. Its
    definition requires two of the axioms of a symmetric space in the
    sense of \cite{MR0239005}. To a given symmetric system $T$ 
    we associate
    the category of $T$-reflection groups. It is an important observation
    that every morphism from one reflection group to another
    constitutes a central extension. The category turns out to be very
    well-behaved: Morphisms are uniquely determined by the reflection
    groups from which and into which they are defined. The
    category can be understood as a complete lattice.
    
    \begin{defn}\propname{Symmetric system}
      Let $T$ be a set with a (not necessarily associative) multiplication
      \begin{align*}
        \mu: T\times T\to T,~
        (s,t)\mapsto s.t.
      \end{align*}
      Then the pair $(T,\mu)$ is called a
      \emph{symmetric system}\index{symmetric system}
      if the following conditions are satisfied for all $s$, $t$ and
      $r\in T$:
      \begin{enumerate}
  \item[(S1)]
        \begin{math}
          s.(s.t)=t,
        \end{math}
  \item[(S2)]
        \begin{math}
          r.(s.t)=(r.s).(r.t).
        \end{math}
      \end{enumerate}
      By abuse of language, 
      we will sometimes say that $T$ is a symmetric system
      instead of saying that $(T,\mu)$ is a symmetric system. If
      $s.t=t$ for all $s$ and $t\in T$ then we call $\mu$ the 
      \emph{trivial multiplication}. If $s$ and $t\in T$ we write
      $s\perp t$ if $s\neq t$, $s.t=t$ and $t.s=s$.
    \end{defn}
    
    \begin{exmp}\label{exmp:symSys}
      Let $\calX$ be a group with a subset $T$ of
      involutions that is invariant under conjugation. Then 
      the multiplication
      \begin{align*}
        T\times T\to T,~(t,s)\mapsto t.s=tst^{-1}
      \end{align*}
      turns $T$ into a symmetric system. We have $s\perp t$ if and
      only if $s$ and
      $t$ are distinct and commute.
    \end{exmp}

    For the remainder of this section, let $T$ be a symmetric system. 

    \begin{defn}\propname{Reflection group}\label{defn:reflGroup}
      Let  $\calX$ be a group acting on $T$. 
      We will denote the element in $T$ obtained by $x$ acting on
      $t$ by $x.t$. 
      Let 
      \begin{equation*}
        \sbullet^\calX:~
        T\to\calX,~t\mapsto t^\calX
      \end{equation*}
      be a function. Then $(\calX,\sbullet^\calX)$ is called a
      $T$-\emph{reflection group}%
      \index{reflection group},
      if the following conditions are satisfied:
      \begin{enumerate}
  \item[(G1)] The group $\calX$ is generated by the set
        $T^\calX:=\{t^\calX~|~t\in T\}$.
  \item[(G2)] For all $s$ and $t\in T$ we have
        \begin{math}
          t^\calX.s=t.s.
        \end{math}
  \item[(G3)] For all $s$ and $t\in T$ we have
        \begin{math}
          t^\calX.
          s^\calX
          ~=~
          (t.s)^\calX.
        \end{math}
  \item[(G4)] For every $t\in T$ we have
        \begin{math}
          (t^\calX)^2
          =
          1.
        \end{math}
      \end{enumerate}
      If we do not need to specify the map $\sbullet^\calX$ we will
      also say that $\calX$ is a reflection group instead of
      saying that $(\calX,\sbullet^\calX)$ is a reflection
      group.
      We say that a reflection group $(\calX,\sbullet^\calX)$
      \emph{separates reflections} 
      if the map $\sbullet^\calX$ is injective and that $\calX$ is
      \emph{proper} if it separates reflections and
      $t^\calX\neq\id$ for every $t\in T$. We say that $T$ itself is
      \emph{proper} if it has a proper reflection group.
    \end{defn}
    
    \begin{exmp}
      In the context of Example~\ref{exmp:symSys}, the group $\calX$
      can be viewed as a proper reflection group for $T$, if it is
      generated by $T^\calX$.
    \end{exmp}

    \begin{rem}\label{rem:proper}
      If $T$ is proper, then we immediately obtain $s.s=s$ and
      \begin{math}
        s.t=t \iff t.s=s
      \end{math}
      for all $s$ and $t\in T$ by considering the images in the proper
      reflection group. In that case $T$ is a discrete symmetric space
      in the sense of \cite{MR0239005}.
    \end{rem}

    \begin{rem}\label{rem:winvariance}
      If $\calX$ is a $T$-reflection group, then
      we have
      \begin{equation*}
        x.(t^\calX)=(x.t)^\calX
      \end{equation*}
      for all $x\in\calX$ and $t\in T$,
      since $T^\calX$ generates $\calX$.
      If $\calY$ is another $T$-reflection group and if 
      \label{rem:relTransit}
      $s_1,\dots,s_n$ and  $t\in T$ then
      \begin{align*}
        (
        s_1^\calX
        \cdots
        s_n^\calX
        ).t
        =
        (s_1^\calY
        \cdots
        s_n^\calY
        ).t.
      \end{align*}
      This means that the $\calX$-orbits and the $\calY$-orbits in $T$
      are the same.
    \end{rem}
    

    \begin{defn}\propname{Reflection morphism}
      Let $T$ be a symmetric system and let 
      $\calX$ and 
      $\calY$ be $T$-reflection
      groups.
      Then a group
      homomorphism $\calX\to\calY$ is called a 
      \emph{$T$-reflection morphism} or a
      \emph{$T$-morphism},
      if the following diagram is commutative: 
      \begin{lastdisplay}
        \begin{center}
          \begin{math}
            \xymatrix{
              &T&\\
              \calX
              \ar[rr]
              \ar@{<-}[ru]
              &&
              \calY.
              \ar@{<-}[lu]
            }
          \end{math}
          \vspace{-0.1cm}~
        \end{center}
      \end{lastdisplay}
    \end{defn}

      If reflection groups $\calX$ and $\calY$ are given, 
      then there is at
      most one morphism $\calX\to\calY$, since the images of a
      generating set are prescribed.
      Every reflection morphism is surjective, since it has a generating
      set in its image. If $\zeta:~\calX\to\calY$ is a
      reflection morphism, then the actions of $\calX$ and
      $\calY$ on $T$ are compatible in the following sense:
      We have
      \begin{math}
        x.t=\zeta(x).t.
      \end{math}
      for all $x\in\calX$ and $t\in T$.
      
    \begin{defn}\propname{Category of $T$-reflection groups}
      Let $T$ be a symmetric system. Then we associate to it the
      \emph{category of $T$-reflection groups} as follows:
      The objects of the category are the $T$-isomorphy classes of
      reflection groups. The morphisms in this category are given by the
      reflection morphisms.
    \end{defn}

    \begin{lem}\label{lem:inBetween}
      Suppose $\calX$ and $\calZ$ are $T$-reflection
      groups. Let $\calY$ be a group and suppose there are group
      homomorphisms
      \begin{equation*}
        \calX
        \overset{\phi}\longrightarrow
        \calY
        \overset{\psi}\longrightarrow
        \calZ,
      \end{equation*}
      such that $\phi$ is surjective and
      the composition $\psi\circ\phi$ is a $T$-reflection
      morphism. Then the map
      \begin{equation*}
        \sbullet^\calY:~
        T\to\calY,~
        t\mapsto\phi(t^\calX)
      \end{equation*}
      turns the pair $(\calY,\sbullet^\calY)$ into a $T$-reflection
      group. 
    \end{lem}

    \begin{proof}
      The homomorphism $\psi$ induces an action of $\calY$ on
      $T$. We will verify the axioms of
      Definition~\ref{defn:reflGroup}. Since $\phi$ is surjective, we
      have (G1). If $s$ and $t\in T$ then
      \begin{equation*}
        t^\calY.s
        ~=~\phi(t^\calX).s
        ~=~\psi\circ\phi(t^\calX).s
        ~=~t^\calZ.s
        ~=~t.s,
      \end{equation*}
      so (G2) is satisfied. Finally, we turn to (G3) and (G4). 
      For all $s,t\in T$
      we have
      \begin{lastdisplay}
      \begin{align*}
        (t^\calY)^2
        &=
        \big(\phi(t^\calX)\big)^2
        =
        \phi\big((t^\calX)^2\big)
        ~=~1
        \text{~~~~~and}
        \\
        t^\calY.s^\calY
        &=
        t^\calY
        s^\calY
        (t^\calY)^{-1}
        =
        \phi(
        t^\calX
        s^\calX
        (t^\calX)^{-1}
        )
        =
        \phi(t^\calX.s^\calX)
        =
        \phi\big((t.s)^\calX\big)
        =
        (t.s)^\calY
      \end{align*}
      \end{lastdisplay}
    \end{proof}

    Though quickly proved, the following Lemma is an important observation about
    $T$-morphisms.

    \begin{lem}\label{lem:centralExt}
      Every $T$-morphism 
      $\zeta:~\calX\to\calY$ is a central extension.
    \end{lem}

    \begin{proof}
      Let $x\in\ker\zeta$. Then we have
      \begin{equation*}
        xs^\calX x^{-1}
        =
        x.s^\calX
        =
        (x.s)^\calX
        =
        (\zeta(x).s)^\calX
        =
        s^\calX.
      \end{equation*}
      for every $s\in T$. This proves the claim.
    \end{proof}




    \begin{defn}\propname{Initial and terminal reflection group}
      A $T$-reflection group $\calU$ 
      is called \emph{initial}
      if the following
      universal property is satisfied: For every $T$-reflection
      group $\calX$ there is a unique reflection morphism
      $\calU\to\calX$.  
      
      A $T$-reflection group $\calT$ 
      is called \emph{terminal} if the following
      universal property is satisfied: For every $T$-reflection
      group $\calX$ there is a unique reflection morphism
      $\calX\to\calT$.  
    \end{defn}

    \begin{prop}\label{prop:initTerm}
      An initial and a
      terminal $T$-reflection group exist. 
    \end{prop}

    \begin{proof}
      For $s\in T$ set
      \begin{equation*}
        s^\calT:~T\to T,~t\mapsto s.t
      \end{equation*}
      and note that $s^\calT$ is its own inverse and thus
      a bijection due to axiom (S1).
      Let $\calT$ be the subgroup of the symmetric group over
      $T$ 
      generated by
      $T^\calT=\{s^\calT~|~s\in T\}$. We will show that the pair
      $(\calT,\sbullet^\calT)$ is a reflection group. There is a natural
      action of $\calT$ on $T$. Axiom (G1) is
      satisfied by definition and if $s$ and $t\in T$ then 
      \begin{math}
        t^\calT.s=t.s
      \end{math}
      by definition, so axiom (G2) is satisfied. Now let $r$, $s$ and
      $t\in T$. Then
      \begin{align*}
        (t^\calT)^2.r
        &=
        t^\calT.(t^\calT.r)
        =
        t.(t.r)
        =
        r
        \text{~~~by (S1)~~~and}
        \\
        (t^\calT.s^\calT).r
        &=
        \big(t^\calT s^\calT(t^\calT)^{-1}\big).r
        =
        t.\big(s.(t.r)\big)
        =
        (t.s).\big(t.(t.r)\big)
        \text{~~~by (S2)}
        \\
        &=
        (t.s).r
        =
        (t.s)^\calT.r,
      \end{align*}
      so axioms (G3) and (G4) are satisfied. Now we show that $\calT$ is
      terminal. If $\calX$ is another reflection group then the action
      of $\calX$ on $T$ induces a group homomorphism into the
      symmetric group on $T$. Its image is in $\calT$ by (G2). This
      group homomorphism is a reflection morphism.
      
      Now consider the group $\calU$ with the presentation
      \begin{equation}      
        \label{eq:preConj}
        \calU:=
        \big\langle(t^\calU)_{t\in T}
        ~\big|~
        (t^\calU)^2=1        
        \text{~~and~~}
        t^\calU s^\calU(t^\calU)^{-1}
        =(t.s)^\calU
        \text{~~~for $s$ and $t\in T$}
        \big\rangle
      \end{equation}
      with the natural map 
      $\sbullet^\calU:~T\to\calU,~t\mapsto t^\calU$. Let $\calX$ be a
      reflection group. Since $(t^\calX)^2=1$ and
      $t^\calX.s^\calX=(t.s)^\calX$ for all $s$ and $t\in T$, there is
      a group homomorphism $\phi:~\calU\to\calX$ such that the
      following diagram commutes:
      \begin{equation*}
        \xymatrix{
          &T&\\
          \calU
          \ar[rr]^\phi
          \ar@{<-}[ru]
          &&
          \calX.
          \ar@{<-}[lu]
        }
      \end{equation*}
      This is true in particular for $\calX:=\calT$, which induces an
      action of
      $\calU$ on $T$. In this way, $\calU$ becomes a reflection group,
      and it is initial.
    \end{proof}

    \begin{exmp}\label{exmp:coxeter}
      If $(\calW,S)$ is a Coxeter system, then $\calW$ is the initial
      reflection group for the symmetric system
      \begin{math}
        T:=W.S:=\{wsw^{-1}~|~s\in S,~ w\in\calW\}
      \end{math}
      in the sense of Example~\ref{exmp:symSys}.
      (See \cite{myWeyl}~Proposition 4.2, for instance.)
    \end{exmp}
    


    \begin{lem}\label{lem:center}
      If $\calX$ is a $T$-reflection group that separates
      reflections, then its center is the kernel of the reflection
      morphism $\calX\to\calT$, where $\calT$ is the terminal
      $T$-reflection group.
    \end{lem}
    \begin{proof}
      Let $\calX$ be a $T$-reflection group that separates
      reflections. Denote its center by
      $\calZ$. Let $\calY=\calX/\calZ$ and denote by
      $\zeta:~\calX\to\calY$ the quotient morphism. Define
      $\sbullet^\calY:~T\to\calY$ as the composition of $\sbullet^\calX$
      with $\zeta$. Now we have
      \begin{math}
        (z.t)^\calX=z.t^\calX=t^\calX
      \end{math}
      for every $z\in\calZ$.
      Since $\calX$ separates
      reflections, we conclude that $\calZ$ acts trivially on $T$. 
      So there is an
      action of $\calY$ on $T$ satisfying $\zeta(x).t=x.t$ for all
      $x\in\calX$ and $t\in T$. This action turns
      $(\calY,\sbullet^\calY)$ into a reflection group and $\zeta$ into a
      reflection morphism. Since $\calX\to\calT$ is a central
      extension, there is a reflection morphism
      $\calT\to\calY$. Since $\calT$ is terminal, it is an
      isomorphism. This entails $\calZ=\ker(\calX\to\calT)$.
    \end{proof}

    Denote the initial
    $T$-reflection group by $\calU$ and the terminal reflection group
    by $\calT$. Denote the kernel of $\calU\to\calT$ by
    $\calA$. Let $\calZ$ be a subgroup of $\calA$.
    Lemmas~\ref{lem:inBetween} and
    \ref{lem:centralExt} taken together yield:

    \begin{thm}\label{thm:category}
      There is one-to-one correspondence between subgroups of the
      abelian group $\calA$ and isomorphy classes of $T$-reflection
      groups. It is given by
      \begin{equation*}
        \calZ\mapsto\calU/\calZ.
      \end{equation*}
      There is a reflection morphism 
      \begin{math}
        \calU/\calZ_1
        \to
        \calU/\calZ_2
      \end{math}
      if and only if $\calZ_1\subseteq\calZ_2$.

      In other words,
      the category of $T$-reflection groups is small and is
      isomorphic to the
      complete lattice of subgroups of $\calA$ viewed as a category in
      the usual way.
    \end{thm}

    \section{Functors on categories of reflection groups}                %

    In this section we introduce the notion of a system morphism from
    one symmetric system to another. Such a morphism gives rise to a
    functor between the corresponding categories. An important
    example of such a functor is the so-called abelianization functor
    that comes from a very natural system morphism that can be defined
    on any symmetric system. The abelianization functor maps a
    reflection group to its abelianization.

    Suppose $(T,\iota_T)$ and $(S,\iota_S)$ are symmetric systems.
    \begin{defn}\propname{System morphism}
      A map 
      $\overline\sbullet:~T\to S,~t\mapsto\overline t$ 
      is called a 
      \emph{symmetric system morphism} or a
      \emph{system morphism},
      if 
      \begin{math}
        \overline{t.s}=\overline{t}.\overline s
      \end{math}
      for
      all $s$ and $t\in T$.
     \end{defn}

     \begin{defn}
      Let $\overline\sbullet:~T\to S,$ be a system morphism. 
      Let $\calX$ be a $T$-reflection group 
      and 
      let $\calP$ be an $S$-reflection group. A group
      homomorphism $\phi:~\calX\to\calP$ is called
      \emph{$\overline\sbullet$-compatible}
      if $\phi(t^\calX)={\overline t}^\calP$ for every $t\in T$,
      i.e. the following diagram commutes:
      \begin{lastdisplay}
        \begin{equation*}
          \raisebox{13mm}{
          \xymatrix{
          T
          \ar@{->}[rr]
          \ar@{->}[d]
          &&
          S
          \ar@{->}[d]
          \\
          \calX
          \ar@{->}[rr]^\phi
          &&
          \calP.
        }
        }
      \end{equation*}
    \end{lastdisplay}
    \end{defn}

    \begin{lem}\label{lem:initial}
      Let $\calU$ be the initial $T$-reflection group, and let $\calV$
      be the initial $S$-reflection group. Then there is a
      $\overline\sbullet$-compatible group homomorphism $\calU\to\calV$. 
    \end{lem}
    
    \begin{proof}
      It
      suffices to show that the assignment
      \begin{equation*}
        \phi:~t^\calU\mapsto\overline t^\calV,
      \end{equation*}
      for every $t\in T$, can be extended to a group homomorphism
      $\calU\to\calV$. This follows from the fact that we
      have
        \begin{align*}
          \big(\phi(t^\calU)\big)^2
          &=
          (\overline t^\calV)^2
          ~=~
          1
          ~~~\text{and}
          \\
          \phi(t^\calU)
          \phi(s^\calU)
          \big(\phi(t^\calU)\big)^{-1}
          \big(\phi((t.s)^\calU)\big)^{-1}
          &=
          \overline{t}^\calV
          \overline s^\calV
          (\overline t^\calV)^{-1}
          (\overline{t.s}^\calV)^{-1}
          \\
          &=
          \overline{t}^\calV
          \overline s^\calV
          (\overline t^\calV)^{-1}
          ((\overline t.\overline s)^\calV)^{-1}
          \\
          &=
          1
        \end{align*}
        for all $s$ and $t\in T$.
    \end{proof}
    \begin{cor}\label{cor:existComp}
      If $\calU$ is the initial $T$-reflection group and $\calP$ is
      any $S$-reflection group, then there is a
      $\overline\sbullet$-compatible $\calU\to\calP$.
    \end{cor}
    \begin{proof}
      Let $\calV$ be the initial $S$-reflection group. 
      Then this follows from diagram chasing in
      \begin{lastdisplay}
        \begin{align*}
          \xymatrix{
            &
            T
            \ar@{->}[rr]
            \ar@{->}[dl]
            &&
            S
            \ar@{->}[dr]
            \ar@{->}[dl]
            \\
            \calU
            \ar@{->}[rr]
            &&
            \calV
            \ar@{->}[rr]
            &&
            \calP.
          }
        \end{align*}
      \end{lastdisplay}
    \end{proof}

    \begin{lem}
      Let $\calX$ be the terminal $T$-reflection group, and let $\calP$
      be the terminal $S$-reflection group. Then there is a
      $\overline\sbullet$-compatible group homomorphism $\calX\to\calP$. 
    \end{lem}
    
    \begin{proof}
      If $t$ and $r\in T$ then
      \begin{align*}
        \overline{t^\calX.r}
        &=
        \overline{t.r}
        =
        \overline{t}.\overline r
        =
        \overline{t}^\calP.\overline r,
      \end{align*}
      in other words, the following diagram commutes:
      \begin{align*}
        \xymatrix{
          T
          \ar[rr]^{t^\calX}
          \ar[d]
          &&
          T
          \ar[d]
          \\
          S
          \ar[rr]^{\overline t^\calP}
          &&
          S.
        }
        \end{align*}
        In view of how terminal reflection groups are constructed in
        the proof of Proposition~\ref{prop:initTerm},
        this means that the assignment 
        \begin{align*}
          T^\calX\to\calP,~
          t^\calX\mapsto\overline t^\calP
        \end{align*}
      can be continued to a group homomorphism $\calX\to\calP$. It is
      compatible. 
    \end{proof}
    \begin{cor}\label{cor:terminal}
      If $\calP$ is the terminal $S$-reflection group and $\calX$ is
      any $T$-reflection group, then there is a
      $\overline\sbullet$-compatible $\calX\to\calP$.
    \end{cor}

    Let $\calX$ be a $T$ reflection group. Let $M$ be the set of all
    $S$-reflection groups $(\calP,\sbullet^\calP)$ modulo equivalence 
    such that there is a compatible
    $\phi_\calP:~\calX\to\calP$.  
    (Due to Theorem~\ref{thm:category}, this is a set.) 
    This set is not empty by Corollary~\ref{cor:terminal}.
    Let $\overline\calX$ be the subgroup of the group 
    \begin{align*}
      \prod_{\calP\in\M}\calP
    \end{align*}
    generated by the elements
    \begin{align*}
      s^{\overline\calX}
      &=
      \prod_{\calP\in M}s^\calP
    \end{align*}
    with $s\in S$. 
    Then $\overline\calX,\sbullet^\calX$ is an $S$-reflection
    group. It is actually the limit of the groups in $M$. The group
    homomorphism 
    \begin{align*}
      \phi:~\calX\to\overline\calX,~
      x\mapsto
      \prod_{\calP\in M}\phi_\calP(x)
    \end{align*}
    is compatible.

    If $\zeta:~\calX\to\calY$ is a $T$-reflection morphism then we
    denote by $\overline\zeta$ the unique $S$-reflection morphism from
    $\overline\calX$ to $\overline\calY$. We have obtained

    \begin{prop}\label{prop:functor}
      If $\overline\sbullet:~T\to S$ is a system morphism, then
      the assignments $\calX\mapsto\overline\calX$ and
      $\zeta\mapsto\overline\zeta$ provide a functor from the category
      of $T$-reflection groups to the category of $S$-reflection
      groups and there is a $\overline\sbullet$-compatible group
      homomorphism $\calX\to\overline\calX$.
    \end{prop}

    \begin{rem}
      By Lemma~\ref{lem:initial}, the functor maps initial reflection
      groups to initial reflection groups.
    \end{rem}

    \begin{lem}\label{lem:kerGen}
      If ~$\overline\sbullet:~T\to S$ is a surjective system morphism,
      then the kernel of $\calX\to\overline\calX$ is 
      the subgroup of $\calX$ generated by
      \begin{align*}
        M:=\{s^\calX t^\calX
        ~|~
      s,t\in T
      \text{~with~}\overline s=\overline t\}.
      \end{align*}
    \end{lem}
    \begin{proof}
      Set $\calK:=\ker(\calX\to\overline\calX)$ and let $\calN$ be the
      subgroup of $\calX$ generated by $M$. Since $M\subseteq\calK$ we
      have $\calN\subseteq\calK$. Note that
      $M$ is invariant under conjugation by elements in $\calX$, so
      $\calN$ is a normal subgroup. 
      Let $\calP:=\calX/\calN$ and let
      $\phi:~\calX\to\calP$ be the quotient morphism. Since 
      $\phi(s^\calX)=\phi(t^\calX)$ for $s$ and $t\in T$ 
      with $\overline s=\overline t$ and since $\overline\sbullet$ is
      surjective the assignment $\overline t\mapsto\phi(t^\calX)$
      provides a function $\sbullet^\calP:~S\to\calP$. 

      Due to Corollary~\ref{cor:terminal} there is an action of
      $\calX$ on $S$. Since $\calN$ acts trivially on $S$, we can
      define an action of $\calP$ on
      $S$ via
      \begin{math}
        \phi(x).s=x.s
      \end{math}
      for $x\in\calX$ and $s\in S$.
      It is a standard computation to verify that $(\calP,\sbullet^\calP)$
      becomes an $S$-reflection group with this action and that $\phi$
      is $\overline\sbullet$-compatible. So there is an $S$-reflection
      morphism $\overline\calX\to\calP$, which means
      $\calK\subseteq\calN$. 
    \end{proof}

    \begin{exmp}
      Let $T$ be a symmetric system and let $S$ be a singleton with
      the trivial multiplication. The only possible map $f:~T\to S$ is
      a system morphism. 
      If $\calU$ is the initial $T$ reflection group, then there
      is a group homomorphism $\det:~\calU\to\{1,-1\}$ sending the
      generators $T^\calU$ to -1.
      Because of this property, we call the functor discussed in this
      example the \emph{determinant functor}.
    \end{exmp}
    
    The functor constructed in the following plays an important role
    in the last section, when our main result is stated.

    \begin{constr}\propname{The abelianization functor}\label{constr:ab}
      Let $T$ be a symmetric system with terminal reflection group
      $\calT$. We denote the orbit space
      by $T^\ab$ and the orbit map by 
      $\sbullet^\ab:~T\to T^\ab,~t\mapsto t^\ab$. 
      With the trivial multiplication the set $T^\ab$ becomes
      a symmetric system 
      and $\sbullet^\ab$ is a system morphism due to
      \begin{align*}
        (s.t)^\ab
        &=
        (s^\calT.t)^\ab
        =
        t^\ab
        =
        s^\ab.t^\ab
      \end{align*}
      for all $s$ and $t\in T$. So, by Proposition~\ref{prop:functor},
      we have a functor $\ab$ from the category of $T$-reflection
      groups to the category of $T^\ab$-reflection groups.

      Let $\calY$ be a $T$-reflection group.
      If $s$ and $t\in T$ then the
      commutator of the generators $s^\calY$ and $t^\calY$ is given by 
      \begin{align*}
        [s^\calY,t^\calY]
        =
        s^\calY
        t^\calY
        (s^\calY)^{-1}
        (t^\calY)^{-1}
        =
        (s^\calY.t^\calY) t^\calY
        =
        (s.t)^\calY t^\calY
      \end{align*}
      So, in view of Lemma~\ref{lem:kerGen} 
      the image $\calY^\ab$ of $\calY$ under the $ab$-functor
      is precisely the abelianization of $\calY$. 
      The initial $T^\ab$-reflection group $\calV$ has the presentation
      \begin{align*}
        \langle (s^\calV)_{s\in T^\ab}
          ~|~
          (s^\calV)^2=1,~
          s^\calV t^\calV
          =
          t^\calV s^\calV
          \text{~~~for $s$ and $t\in T^\ab$}
            \rangle.
      \end{align*}
      So it is isomorphic to $\Z_2^{T^\ab}$ and it is a proper
      reflection group.
    \end{constr}
    \begin{cor}
      If $\calU$ is the initial $T$-reflection group, then 
      $T^\calU$ is a minimal $\calU$-invariant generating set for it. 
    \end{cor}
    \begin{defn}\propname{The abelianization functor}
      The functor discussed in Construction~\ref{constr:ab} 
      is called the
      \emph{abelianization functor}.
    \end{defn}
    \begin{rem}
      If $~\overline\sbullet:~T\to S$ is a surjective system morphism, 
      then there is
      a unique system morphism $T^\ab\to S^\ab$ such that the
      following diagram commutes:
      \begin{lastdisplay}
        \begin{center}
          \begin{math}
            \xymatrix{
              T
              \ar[rr]
              \ar[d]
              &&
              T^\ab
              \ar[d]
              \\
              S
              \ar[rr]
              &&
              S^\ab.
            }
          \end{math}          
          \vspace{-0.1cm}~
        \end{center}
      \end{lastdisplay}
    \end{rem}

    Let $\calX$, $\calY$, $\calP$ and $\calQ$ be groups and suppose
    we have the commutative diagram
    \begin{equation}\label{eq:preExact}
      \xymatrix{
        \calX
          \ar[rr]
          \ar[d]
          &&
          \calY
          \ar[d]
          \\
          \calP
          \ar[rr]
          &&
          \calQ,
        }
      \end{equation}
      where all the arrows represent surjective group homomorphisms.
      We will use the notation 
      \begin{align}\label{eq:kNotation}
        \calK^\calY_\calQ=\ker(\calY\to\calQ)
      \end{align}
      and omit the subscript, if no confusion arises. The following
      lemma about short exact sequences will be needed later on:

      \begin{lem}\label{lem:exact}
        Diagram (\ref{eq:preExact}) can be extended to the following
        diagram with three vertical and three horizontal short exact sequences:
      \begin{equation*}
        \xymatrix{
          &
          1
          \ar[d]
          &&
          1
          \ar[d]
          &&
          1
          \ar[d]
          \\
          1
          \ar[r]
          &
          \ker(
          \calK^\calX_\calP
          \to
          \calK^\calY_\calQ
          )
          \ar[rr]
          \ar[d]
          &&
          \calK^\calX_\calP
          \ar[rr]
          \ar[d]
          &&
          \calK^\calY_\calQ
          \ar[r]
          \ar[d]
          &
          1
          \\
          1
          \ar[r]
          &
          \ker(\calX\to\calY)
          \ar[rr]
          \ar[d]
          &&
          \calX
          \ar[rr]
          \ar[d]
          &&
          \calY
          \ar[r]
          \ar[d]
          &
          1
          \\
          1
          \ar[r]
          &
          \ker(\calP\to\calQ)
          \ar[rr]
          \ar[d]
          &&
          \calP
          \ar[rr]
          \ar[d]
          &&
          \calQ
          \ar[r]
          \ar[d]
          &
          1
          \\
          &
          1
          &&
          1
          &&
          1.
        }
      \end{equation*}
    \end{lem}
    \begin{proof}
      By
      \cite{bou_algebra}~{\sc Ch~i}~\S~1.4~Proposition~2 the sequence
      \begin{align*}
        1
        \to
        \ker(\calK^\calX_\calP\to\calK^\calY_\calQ)
        \to
        \ker(\calX\to\calY)
        \to
        \ker(\calP\to\calQ)
        \to
        1
      \end{align*}
      is part of a longer exact sequence.
      Although that proposition is stated for abelian groups,
      the proof works for arbitrary groups, as well.
    \end{proof}

    Let $S$ and $T$ be symmetric systems and let
    $\overline\sbullet:~T\to S$ be a system morphism.
    \begin{defn}\propname{System section}
      A system morphism
      $\dot\sbullet:~S\to T$ 
      is called
      a \emph{system section (for $\overline\sbullet$)}, 
      if it is a right inverse of 
      $\overline\sbullet$, (i.e. $\overline{\dot s}=s$ for all $s\in
      S$). If such a system section is given and $t\in T$, we will
      sometimes write
      $\dot t$ instead of $\dot{\overline t}$ to keep notation
      simple. 
    \end{defn}

    \begin{rem}\label{rem:section}
      Suppose $\dot\sbullet:S\to T$ is a system section for
      $\overline\sbullet:T\to S$.  Let $\calU$ be the initial
      $T$-reflection group and let $\calV$ be the initial
      $S$-reflection group. If there is a $\overline\sbullet$-compatible
      group homomorphism $\calX\to\calV$, then the diagram
      \begin{align*}
        \xymatrix{ & S \ar[rr]^{\dot\sbullet} \ar[dl] && T
          \ar[rr]^{\overline\sbullet} \ar[dl] \ar[dr] && S \ar[dr]
          \\
          \calV \ar[rr] && \calU \ar[rr] && \calX \ar[rr] && \calV }
      \end{align*}
      shows that there is a homomorphic section $\calV\to\calX$ for
      $\calX\to\calV$. With the notation established in
      (\ref{eq:kNotation}) 
      this means $\calX\cong\calK^\calX_\calV\rtimes\calV$.
      We will set
      \begin{align*}
        t^{\calK^\calX_\calV}:=t^\calX\dot t^\calX
        \in\calK^\calX_\calV.
      \end{align*}
      In this way we have
      \begin{math}
        t^\calX\cong(t^{\calK^\calX_\calV},\overline t^\calV).
      \end{math}
    \end{rem}

    \section{Symmetric systems extended by an abelian group}             %

    For the entire section let $S$ and $T$ be symmetric systems and
    let $\overline\sbullet:~T\to S$ be a system morphism with a system section
    $\dot\sbullet$. Denote the initial $T$-reflection group by $\calU$ and
    the initial $S$-reflection group by $\calV$.    
    Suppose that the terminal $T$ reflection group $\calA$ satisfies 
    $\overline\calA=\calV$
    and suppose that $\calK:=\ker(\calA\to\calV)$ is 
    abelian. In other words $\calA$ is 
    an extension of the $S$-reflection group $\calV$ by the abelian
    group $\calK$. We will write the group $\calK$ additively.

    Note that the hypotheses imply that
    $\overline\calX=\calV$ for any $T$-reflection group $\calX$. 
    The group homomorphism
    $\calX\to\calV$ has a section 
    $\calV\to\calA$ by Remark~\ref{rem:section}. So $\calX$ is
    isomorphic to the semidirect product $\calK^\calX\rtimes\calV$.

    Let $\calX$ be a $T$-reflection group.
    Set
    \begin{equation*}
      \calX_\ab
      :=
      \calX/(\calK^\calX)'
      ~~~\text{and}~~~ 
      \sbullet^{\calX_\ab}:~
      T\to\calX_\ab,~
      t\mapsto t^{\calX_\ab}=q(t^\calX),
    \end{equation*}
    where $q:~\calX\to\calX_\ab$ denotes the quotient map.
    \begin{lem}
      The pair $(\calX_\ab,\sbullet^{\calX_\ab})$ is a $T$-reflection group.
    \end{lem}
    \begin{proof}
      By Lemma~\ref{lem:exact}  
      we have the following commutative diagram:
      \begin{equation*}
        \xymatrix{
          \ker(\psi)
          \ar[rr]
          \ar[d]_\id
          &&
          \calK^\calX
          \ar[rr]
          \ar[d]
          &&
          \calK^\calA
          \ar[d]
          \\
          \ker(\psi)
          \ar[rr]
          &&
          \calX
          \ar[rr]_\psi
          &&
          \calA
        }
      \end{equation*}
      Since $\calK^\calA$ is abelian, the kernel $\ker\psi$ contains
      the commutator subgroup $(\calK^\calX)'$. So $\psi$ factors through
      $\calX_\ab$ yielding group homomorphisms
      \begin{equation*}
        \calX
        \longrightarrow
        \calX_\ab
        \longrightarrow
        \calA.
      \end{equation*}
      We are done by Lemma~\ref{lem:inBetween}. 
    \end{proof}

    Let $\calB$ be a $T$-reflection group such that $\calB=\calB_\ab$,
    i.e. such that
    $\calK^\calB:=\calK_\calV^\calB$ is abelian. We have the following
    commuting diagram:
      \begin{equation*}
        \xymatrix{
          \ker(\calB\to\calA)
          \ar[rr]
          \ar[d]_\phi
          &&
          \calB
          \ar[rr]
          \ar[d]
          &&
          \calA
          \ar[d]
          \\
          \ker(\calB^\ab\to\calA^\ab)
          \ar[rr]
          &&
          \calB^\ab
          \ar[rr]
          &&
          \calA^\ab
        }
      \end{equation*}

    \begin{prop}\label{prop:abKerIso}
      If $\calK$ is free abelian, then
      $\phi$ is an isomorphism.
    \end{prop}
    
    \begin{proof}
      The central extension $\calK^\calB\to\calK$ splits, since
      $\calK$ is free abelian and $\calK^\calB$ is abelian. So we have
      $\calK^\calB\cong\calZ\times\calK$ for some abelian group
      $\calZ$. By Remark~\ref{rem:section}, we have
      $\calB\cong\calK^\calB\rtimes\calV$ and
      $\calA\cong\calK\rtimes\calV$. We can identify
      $\ker(\calA\to\calB)$ with $\calZ$. Since $\calZ$ is central in
      $\calB$, the action of $\calV$ on it is trivial. This implies
      \begin{align*}
        \calB
        &\cong
        (Z\times\calK)\rtimes\calV
        \cong
        Z\times(\calK\rtimes\calV)
        \cong
        \calZ\times\calA.
      \end{align*}
      So the diagram above becomes
      \begin{equation*}
        \xymatrix{
          \calZ
          \ar[rr]
          \ar[d]_\id
          &&
          \calZ\times\calA
          \ar[rr]
          \ar[d]
          &&
          \calA
          \ar[d]
          \\
          \calZ
          \ar[rr]
          &&
          \calZ\times\calA^\ab
          \ar[rr]
          &&
          \calA^\ab
        }
      \end{equation*}
      The identity map $\id:\calZ\to\calZ$ is the unique map that
      makes this diagram commute.  
    \end{proof}

    \section{Extensions using bihomomorphisms}                           %

    We continue to work in the setting of the previous section.
    In this section we will see how certain bihomomorphisms allow us to
    construct new reflection groups from $\calA$. This
    process could be investigated more generally using arbitrary
    cocycles. However the
    alternating bihomomorphic ones are the ones that we need 
    in order to construct the Weyl group in the last section.

    Suppose that $\calZ$ is an abelian group and that 
    $c:~
    \calK^2
    \to
    \calZ
    $
    is an alternating bihomomorphism that is $\calV$-invariant in the
    following sense: If $k$, $k'\in\calK$ and $v\in\calV$, then 
    \begin{math}
      c(v.k,v.k')=c(k,k').
    \end{math}
    Now consider the group  
    $\calK^c:=\calZ\times_c\calK$ 
    with the multiplication
    \begin{align*}
      (\calK^c)^2
      \to
      \calK^c,~
      \big((z_1,k_1),(z_2,k_2)\big)
      \mapsto
      (z_1+z_2+c(k_1,k_2),k_1k_2)
    \end{align*}
    and note that $\calV$ acts on it by group automorphisms via
    \begin{equation*}
      v.(z,k)=(z,v.k)
    \end{equation*}
    for $v\in\calV$, $z\in\calZ$ and $k\in\calK$. 
    So we can form the semidirect product
    $\calK^c\rtimes\calV$.
    Set
    \begin{equation*}
      \sbullet^\calX:~
      T\to\calK^c\rtimes\calV,~
      t\mapsto t^\calX
      =(0,t^\calK,\overline t^\calV)
    \end{equation*}
    and let $\calX$ be the subgroup of $\calK^c\rtimes\calV$ generated
    by $T^\calX$.
    The projection 
    $
    \calK^c\rtimes\calV
    \to
    \calK\rtimes\calV
    $
    is a group homomorphism which restricts to one 
    $\calX\to\calA$ such that $t^\calX$
    is mapped to $t^\calA$ for every $t\in T$. 
    Via this homomorphism we obtain an action
    of $\calX$ on $T$.

    \begin{defn}
      We will denote the pair 
      $(\calX,\sbullet^\calX)$ by $(\calA^c,\sbullet^{\calA^c})$ and
      call $c$
      a \emph{reflection bihomomorphism (for $\calA$)} 
      if
      $\calA^c$ is a $T$-reflection group. 
    \end{defn} 

    If $c$ is a reflection bihomomorphism, 
    then $\calA^c\to\calA$ is a reflection
    morphism. So
    $\overline{\calA^c}=\calV$ 
    and 
    $\calK_\calV^{\calA^c}$
    is the subgroup of $\calK^c$ generated by $\{0\}\times_cT^\calK$.

    \begin{lem}\label{lem:reflBihom}
      An alternating bihomomorphism 
      $c:~\calK^2\to\calZ$ is a reflection bihomomorphism if
      and only if 
      $
      c(
      s^\calV.t^\calK,t^\calK
      )
      =0
      $
      for every $s$ and $t\in T$.
    \end{lem}

    \begin{proof}
      Set $\calX:=\calY^c$ and note that
      \begin{align*}
        \overline t^\calV.t^\calK
        =
        \dot t^\calX 
        t^\calX
        \dot t^\calX
        \dot t^\calX
        =
        (t^\calK)^{-1}
      \end{align*}
      in $\calA$ for every $t\in T$. Written additively in $\calK$
      this means
      \begin{math}
        \overline t^\calV.t^\calK
        =
        -t^\calK.
      \end{math}
      Axioms (G1) and (G2) are satisfied. If $s$ and $t\in T$ then
      \begin{align*}
        (t^{\calX})^2
        &=
        \big(0,t^\calK,\overline t^\calV\big)^2
        =
        \big(
        c(t^\calK,
        \overline t^\calV.t^\calK),
        t^\calK
        +\overline t^\calV.t^\calK,
        (\overline t^\calV)^2\big)
        \\
        &=
        \big(
        -c(
        t^\calK,
        t^\calK
        ),0,1
        \big)
        =
        (0,0,1)
        \hspace{1cm}\text{and}
        \\
        t^\calX.s^\calX
        &=
        t^\calX
        s^\calX
        t^\calX
        =
        (0,t^\calK,\overline t^\calV)
        (0,s^\calK,\overline s^\calV)
        (0,t^\calK,\overline t^\calV)
        \\
        &=
        \big(
        c(t^\calK,\overline t^\calV.s^\calK),
        t^\calK+\overline t^\calV.s^\calK,
        \overline t^\calV\overline s^\calV
        \big)
        (0,t^\calK,\overline t^\calV)
        \\
        &=
        \Big(
        c(\overline t^\calV.t^\calK,s^\calK)
        +
        c\big(
        t^\calK
        +
        \overline t^\calV.s^\calK,
        (\overline t^\calV\overline s^\calV).t^\calK
        \big),
        \underbrace{*,*}_{(t.s)^\calA}
        \Big)
        \\
        &=
        \big(
        c(-t^\calK,s^\calK)
        +
        c(
        s^\calV.t^\calK-s^\calK,
        t^\calK
        ),
        (t.s)^\calK,
        \overline{t.s}^\calV
        \big)
        \\
        &=
        \big(
        c(s^\calV.t^\calK,t^\calK),
        (t.s)^\calK,
        \overline{t.s}^\calV
        \big).
      \end{align*}
      So, axioms (G3) and (G4) are satisfied if and only if 
      $c(s^\calV.t^\calK,t^\calK)=0$.
    \end{proof}

    \begin{rem}\label{rem:comMap}
      Let $\calX$ be an arbitrary $T$-reflection group and
      let $\zeta:~\calX\to\calA$ be a reflection morphism. We have the
      central extension $\zeta:~\calK^\calX\to\calK$.
      There is a unique map
      \begin{math}
        c_\calX:~
        \calK^2
        \to
        (\calK^\calX)'
      \end{math}
      satisfying
      \begin{math}
        c_\calX\big(\zeta(k),\zeta(l)\big)=[k,l]
      \end{math}
      for all $k$ and $l\in\calK$.
      This map is bilinear since $\calK^\calX$ is two-step nilpotent.
      It is also $\calV$-invariant and alternating. 
    \end{rem}
    Suppose $\calX=\calA^b$ for a reflection bihomomorphism
    $b:~\calK^2\to\calZ$. 
    \begin{lem}\label{lem:commutators}
      There is an injective group homomorphism
      $(\calK^\calX)'\to2\calZ$ such that the following diagram
      commutes: 
      \begin{align*}
        \xymatrix{
        \calK^2
        \ar[rrrr]^{c_\calX}
        \ar[rrd]_b
        &&&&
        (\calK^\calX)'
        \ar[d]
        \\
        &&
        \calZ
        \ar[rr]
        &&
        2\calZ
      }
      \end{align*}
    \end{lem}
    
    \begin{proof}
      For $(z_1,k_1)$ and $(z_2,k_2)\in\calK^\calX$,
      we can compute the
      commutator in $\calK^\calX$ as follows:
      \begin{lastdisplay}
      \begin{align*}
        &
        [(z_1,k_1),(z_2,k_2)]
        =
        (z_1,k_1)
        (z_2,k_2)
        (-z_1,-k_1)
        (-z_2,-k_2)
        \\
        &=
        (z_1+z_2+c(k_1,k_2),k_1+k_2)
        (-z_1-z_2+c(-k_1,-k_2),-k_1-k_2)
        \\
        &=
        (2c(k_1,k_2),0).
      \end{align*}
      \end{lastdisplay}
    \end{proof}
    \begin{cor}\label{cor:comMap}
      If $\calZ$ contains no involutions, then there 
      is an injective group homomorphism
      $(\calK^\calX)'\to\calZ$ such that the following diagram
      commutes: 
      \begin{align*}
        \xymatrix{
          \calK^2
          \ar[rr]
          \ar[rrd]
          &&
          (\calK^\calX)'
          \ar[d]
          \\
          &&
          \calZ.
        }
      \end{align*}
    \end{cor}
    Now let $\calX$ be an arbitrary $T$-reflection
    group. 
    The following Proposition is key observation that makes a connection
    between the bihomomorphism $c_\calX$ and reflection
    bihomomorphisms. 
    
    \begin{prop}\label{prop:commutatorMap}
      If $\calA$ separates reflections, 
      then the map $c_\calX$ is a reflection bihomomorphism.
    \end{prop}
    
    \begin{proof}
      The map $c_\calX$ is bihomomorphic since $\calK^\calX$ is 2-step
      nilpotent. It is alternating and $\calV$-invariant. 
      Now let $s$ and 
      $t\in T$. Since $\calK$ is abelian, we have
      \begin{align*}
        \dot s^\calA
        &=
        \big[
        t^\calK,
        s^\calV.t^\calK
        \big]
        \dot s^\calA
        =
        \big[
        t^\calA
        \dot t^\calA,
        \dot s^\calA.\big(
        t^\calA
        \dot t^\calA
        \big)
        \big]
        \dot s^\calA
        \\
        &=
        t^\calA
        \dot t^\calA
        \dot s^\calA
        t^\calA
        \dot t^\calA
        \dot s^\calA
        \dot t^\calA
        t^\calA
        \dot s^\calA
        \dot t^\calA
        t^\calA
        \dot s^\calA
        \dot s^\calA
        =
        (
        t^\calA
        \dot t^\calA
        \dot s^\calA
        t^\calA
        \dot t^\calA
        ).
        \dot s^\calA.
      \end{align*}
      Since $\calA$ separates reflections, this entails
      \begin{math}
        \dot s
        =        
        (
        t^\calA
        \dot t^\calA
        \dot s^\calA
        t^\calA
        \dot t^\calA
        ).
        \dot s.
      \end{math}
      Using Remark~\ref{rem:winvariance} 
       we obtain
      \begin{align*}
        c_\calX\big(
        t^\calK,
        s^\calV.t^\calK
        \big)
        &=
        \big[
        t^\calX
        \dot t^\calX,
        \hat s.\big(
        t^\calX
        \dot t^\calX
        \big)
        \big]
        =
        \big(
        (
        t^\calX
        \dot t^\calX
        \dot s^\calX
        t^\calX
        \dot t^\calX
        ).
        \dot s^\calX
        \big)
        \dot s^\calX
        \\
        &=
        \big(
        (
        t^\calX
        \dot t^\calX
        \dot s^\calX
        t^\calX
        \dot t^\calX
        ).
        \dot s\big)^\calX
        \dot s^\calX
        =
        \big(
        (
        t^\calA
        \dot t^\calA
        \dot s^\calA
        t^\calA
        \dot t^\calA
        ).
        \dot s\big)^\calX
        \dot s^\calX
        =
        \dot s^\calX
        \dot s^\calX
        =
        1.          
      \end{align*}
      We are done by Lemma~\ref{lem:reflBihom}.
    \end{proof}

    \section{Finite root systems}                                        %

    In this section we investigate the action of the Weyl group
    $\calV$ of an irreducible finite root system $\Delta$ on the root
    lattice $\calL$ and the coroot lattice $\check\calL$. 
    We will derive results needed in
    the last section of this article. All of these results are stated
    in terms of the action of $\calV$ on the abelian groups
    $\calL$ and $\check\calL$ and
    their subsets $\Delta$ and $\check\Delta$ (the coroot system) and
    rely on the pairing
    $\check\calL\times\calL\to\Z$ and the
    bijection $\Delta\to\check\Delta$.

    Throughout this section, 
    let $\Delta$ be an irreducible finite root system with root basis
    $B$. 
    If $\alpha$ and $\beta$ are distinct roots in $B$,
    then we will say that $\alpha$ and
    $\beta$ are \emph{adjacent} and write $\alpha\sim\beta$ 
    if they are connected by an
    edge in the Dynkin diagram. We will say that $\alpha$ and $\beta$
    are
    \emph{simply adjacent} if they are
    connected by a simple edge. We will write
    $\alpha\not\sim\beta$ if $\alpha$ and $\beta$ are not connected by an
    edge.

    The root system $\Delta$ is the disjoint union of the divisible
    roots $\Delta_\ex$ (for extra long)
    and the indivisible ones $\Delta_\red$ (for reduced). The set
    $\Delta_\red$, in turn, can be partitioned into the short roots
    $\Delta_\sh$ and the long roots $\Delta_\lng$.
    The coroots $\check\Delta$ form a root system in their own right
    and we partition them in the same way: 
    $
    \check\Delta
    =
    \check\Delta_\sh
    \cup
    \check\Delta_\lng
    \cup
    \check\Delta_\ex.
    $
    Denote the root lattice by $\calL$ and the coroot lattice by
    $\check\calL$. There is a bihomomorphic pairing
    \begin{math}
      \check\calL\times\calL
      \to
      \Z,~
      (\lambda,\mu)\mapsto
      \langle\lambda,\mu\rangle
    \end{math}
    and a bijection
    \begin{math}
      \check\sbullet,~
      \Delta\to\check\Delta,~
      \alpha\mapsto\check\alpha.
    \end{math}    For every $\alpha\in\Delta$ there
    is a reflection $r_\alpha$ that acts in the following ways on
    $\calL$ and $\check\calL$:
    \begin{align*}
      r_\alpha.\lambda=\lambda-\langle\check\alpha,\lambda\rangle\alpha
      \text{~~~and~~~}
      r_\alpha.\mu=\mu-\langle\mu,\alpha\rangle\check\alpha
    \end{align*}
    for $\lambda\in\calL$ and $\mu\in\check\calL$. 
    The Weyl group $\calV$ is the subgroup of $\Aut(\calL)$ generated
    by the reflections $r_\Delta$. It acts faithfully and by group
    automorphisms on $\calL$ and
    $\check\calL$. We can view $\calL$ and $\check\calL$ as
    $\calV$-modules. The bijection $\Delta\to\check\Delta$ is
    $\calV$-equivariant and the pairing $\langle\sbullet,\sbullet\rangle$ is
    $\calV$-invariant in the following sense:
    We have     
    \begin{math}
      \langle v.\lambda,v.\mu\rangle
      =
      \langle\lambda,\mu\rangle
    \end{math}
    for $\lambda\in\calL$, $\mu\in\check\calL$ and $v\in\calV$.
    We provide some
    important values of the pairing for root systems of rank 2
    in Table~\ref{tab:rootSystems} using the standard terminology of
    the classification of the irreducible finite root systems.    
    \begin{table}[t]
      \centering
      \begin{tabular}{c|c|c|c|c|c}
        \bf Type
        &
        \bf Dynkin Diagram
        &
        $\langle\check\alpha,\beta\rangle$
        &
        $\langle\check\beta,\alpha\rangle$
        &
        $r_\alpha.\beta$
        &
        $r_\beta.\alpha$
        \\
        \hline
        $A_2$
        &
        $
        \alpha~   
        {\circ} 
        {\put(-1,3){\line(1,0){28}}}
        {\hspace{0.3cm}}
        \phantom{>}
        {\hspace{0.3cm}}
        {\circ}
        ~\beta
        $
        &
        -1
        &
        -1
        &
        $\beta+\alpha$
        &
        $\alpha+\beta$
        \\
        $B_2$
        &
        $
        \alpha~   
        {\circ} 
        {\put(-1,2){\line(1,0){28}}}
        {\put(-1,4){\line(1,0){28}}}
        {\hspace{0.3cm}}
        {<}
        {\hspace{0.3cm}}
        {\circ}
        ~\beta
        $
        &
        -2
        &
        -1
        &
        $\beta+2\alpha$
        &
        $\alpha+\beta$
        \\
        $G_2$
        &
        $   
        \alpha~   
        {\circ} 
        {\put(-2,1.2){\line(1,0){30}}}
        {\put(-1,3){\line(1,0){28}}}
        {\put(-2,4.8){\line(1,0){30}}}
        {\hspace{0.3cm}}
        {<}
        {\hspace{0.3cm}}
        {\circ}
        ~\beta
        $
        &
        -3
        &
        -1
        &
        $\beta+3\alpha$
        &
        $\alpha+\beta$
      \end{tabular}
      \caption{Irreducible finite root systems of rank 2}
      \label{tab:rootSystems}
    \end{table}

    If we set
    \begin{math}
      T:=\{r_\alpha~|~\alpha\in\Delta\},
    \end{math}
    then $T$ is a symmetric system with the multiplication defined
    by $t.s=tst^{-1}$ for $s$ and $t\in T$. 
    \begin{defn}
      We call $T$ 
      \emph{the symmetric system associated to $\Delta$.}
    \end{defn}
    If
    \begin{math}
      S:=\{r_\alpha~|~\alpha\in B\},
    \end{math}
    then $(\calV,S)$ is a Coxeter system. As discussed in
    Example~\ref{exmp:coxeter}, the group $\calV$ is the initial
    $T$-reflection group. Recall that for $\alpha$ and
    $\beta\in\Delta$ we write $r_\alpha\perp r_\beta$ if
    $r_\alpha\neq r_\beta$ and $[r_\alpha,r_\beta]=1$.  

    \begin{rem}\label{rem:calLId}
      If $\Delta$ is reduced and simply laced,
      i.e. $\Delta_\lng=\Delta_\ex=\emptyset$, then the map
      $\Delta\to\check\Delta$ extends to a $\calV$-equivariant group
      isomorphism $\calL\to\check\calL$. So we can identify the
      $\calV$-modules 
      $\calL$ and $\check\calL$.
      If $\Delta$ is non-reduced,
      i.e. $\Delta_\ex\neq\emptyset$, then there is an
      $\calL$-equivariant isomorphism $\calL\to\check\calL$ that maps
      $\Delta$ in the following way:
      \begin{align*}
        \alpha\mapsto
        \begin{cases}
          \check\alpha&\text{if $\alpha\in\Delta_\lng$},
          \\
          \frac12\check\alpha&\text{if $\alpha\in\Delta_\sh$}, 
          \\
          2\check\alpha&\text{if $\alpha\in\Delta_\ex$}.
        \end{cases}
      \end{align*}
      Again, we may identify the $\calV$-modules
      $\calL$ and $\check\calL$.
    \end{rem}  
    
    Now suppose $\Delta$ is reduced and non-simply laced,
    i.e. $\Delta_\ex=\emptyset$ and $\Delta_\lng\neq\emptyset$.
    Then 
    $\Delta_\sh$ is mapped to $\check\Delta_\lng$ and
    $\Delta_\lng$ is mapped to $\check\Delta_\sh$ by 
    $\Delta\to\check\Delta$.
    Set
    \begin{align}\label{eq:k}
      k_\Delta:=
      \begin{cases}
        2&\text{if $\Delta$ is of type 
          $B_\ell(\ell\ge2)$,
          $C_\ell(\ell\ge3)$,
          $F_4$ or
          $BC_\ell(\ell\ge2)$,}
        \\
        3&\text{if $\Delta$ is of type 
          $G_2$.}
      \end{cases}
    \end{align}
    \begin{lem}\label{lem:nonsimplyId}
      There is an
      injective $\calV$-equivariant group homomorphism from
      $\calL\to\check\calL$ satisfying
      \begin{align*}
        \phi(\alpha)
        =
        \begin{cases}
          \check\alpha&\text{if $\alpha\in\Delta_\lng$} 
          \\
          k_\Delta\check\alpha&\text{if $\alpha\in\Delta_\sh$} 
        \end{cases}
      \end{align*}
      Via $\phi$ the $\calV$-module
      $\calL$ can be identified with the submodule of $\check\calL$
      generated by $\check\Delta_\sh$. With this identification we
      have
      $k_\Delta\check\calL\subseteq\calL\subseteq\check\calL$.
     \end{lem}

    \begin{proof}
      The groups $\calL$ and
      $\check\calL$ can be viewed as subgroups of a real vector space
      $V$ and its dual $V^*$ respectively. According to 
      \cite{bou}~Ch.~{\sc vi},%
      ~$\mathrm n^{\mathrm o}$~1.1~Proposition~3
      and
      \cite{bou}~Ch.~{\sc vi},%
      ~$\mathrm n^{\mathrm o}$~1.4~Proposition~11
      we can
      pick a $\calV$-invariant symmetric non-degenerate bilinear form
      $(\sbullet|\sbullet)$ on $V$ that satisfies $(\beta,\beta)=2$ for
      all $\beta\in\Delta_\lng$. If we use this form 
      to identify $V$ with $V^*$, then we have
      \begin{align*}
        \check\alpha=\frac{2\alpha}{(\alpha|\alpha)}
      \end{align*}
      for every $\alpha\in\Delta$. If $\alpha\in\Delta_\sh$ and
      $\beta\in\Delta_\lng$ then      
      \begin{align*}
        \frac{(\beta|\beta)}{(\alpha|\alpha)}
        =
        \frac{\langle\check\alpha,\beta\rangle}%
        {\langle\check\beta,\alpha\rangle}
        =k_\Delta,
        \text{~~so~~}
        \check\alpha=\frac{2\alpha}{(\alpha|\alpha)}=k_\Delta\alpha
        \text{~~and~~}
        \check\beta=\frac{2\beta}{(\beta|\beta)}=\beta.
      \end{align*}
      The last statement in the Lemma follows from the fact that the
      root lattice of a root system is generated by the short roots. 
    \end{proof}

    \begin{rem}\label{rem:nonReduced}
      If $\Delta$ is non-reduced, i.e. is of type
      $BC_\ell(\ell\ge1)$ then $\Delta_\red$ is a reduced irreducible root
      system. (See for instance
      \cite{bou}~Ch.~{\sc vi},~$\mathrm{n}^\mathrm{o}$~4.14.)
      The Weyl
      group and the root lattice remain unchanged when passing from
      $\Delta$ to $\Delta_\red$ and the following changes in types can
      occur: 
      \begin{lastdisplay}
      \begin{equation*}
        BC_\ell\mapsto
        \begin{cases}
          A_1&\text{if $\ell=1$}\\
          B_\ell&\text{if $\ell\ge2$}.
        \end{cases}
      \end{equation*}
      \end{lastdisplay}
    \end{rem}

    Set
    \begin{equation*}
      \calL_\eff
      =
       \langle v.l-l~|~v\in\calV,l\in\calL\rangle.
    \end{equation*}

    \begin{prop}\label{prop:calLEff}
      \begin{equation*}
        \frac{
          \calL
        }{
          \calL_\eff
        }
        \cong
        \begin{cases}
          \Z_2
          &
          \text{
            if $\Delta$ is of type 
            $A_1$,
            $B_\ell(\ell\ge2)$ or
            $BC_\ell(\ell\ge1)$},
          \\
          \{0\}
          &
          \text{otherwise}.
        \end{cases}
      \end{equation*}
      In the case of $\calL/\calL_\eff\cong\Z_2$ 
      the quotient map $\calL\to\Z_2$ satisfies
      \begin{equation*}
        \alpha\mapsto
        \begin{cases}
          1&\text{if $\alpha$ is short or the type is $A_1$}\\
          0&\text{otherwise}
        \end{cases}
      \end{equation*}
      for every $\alpha\in\Delta$. 
    \end{prop}

    \begin{proof}
      In view of Remark~\ref{rem:nonReduced}, it suffices to consider
      reduced root systems $\Delta$.
      The following is true if $\Delta$ is a
      root system of any type except
      $A_1$ or $B_\ell(\ell\ge2)$: 
      If $\alpha\in B$ then there
      is a root $\beta\in R$ such that
      $r_\alpha.\beta=\beta+\alpha$,
      i.e. $\langle\check\alpha,\beta\rangle=-1$. 
      We will show this by using the classification of the finite
      irreducible root systems and the information from
      Table~\ref{tab:rootSystems}.  In the cases
      $A_\ell(\ell\ge 2)$, 
      $D_\ell(\ell\ge 4)$,
      $E_6$, $E_7$, $E_8$ and
      $F_4$ every vertex of the Dynkin diagram is connected to another
      vertex by a simple edge.
      So we can take $\beta$ to be in $B$ and simply
      adjacent to $\alpha$.
      In the
      case of type $C_\ell(\ell\ge3)$ and $\alpha$ long we may use
      the only adjacent root $\beta$. To understand the case $G_2$,
      let $B=\{\alpha,\beta\}$ with $\alpha$ short and $\beta$
      long. Then we have
      \begin{eqnarray*}
        r_\alpha.(\beta+\alpha)&=&\beta+2\alpha
        \text{~~~and}
        \\
        r_\beta.\alpha&=&\alpha+\beta.
      \end{eqnarray*}
      We may conclude
      \begin{math}
        \calL_\eff
        =
        \calL
      \end{math}
      and thus $\calL/\calL_\eff=\{0\}$.

      Now we look at the construction of $B_\ell$ given in 
      \cite{bou}~Ch.~{\sc vi},~$n^0$~4.5. By allowing $\ell=1$ we also
      cover the case of $A_1$. 
      So let
      $\epsilon_1,\dots\epsilon_\ell$ be a basis of $\Z^n$. 
      The short roots in $\Delta$ are 
      $\pm\epsilon_i~(i=1,\dots,\ell)$ and the long roots are
      $\pm\epsilon_i\pm\epsilon_j~(1\le i<j\le\ell)$. The root lattice
      $\calL$ is $\Z^n$. 
      According to
      \cite{bou}~Ch.~{\sc vi},~$n^0$~1.4 Proposition~11, the Weyl
      group $\calV$ acts transitively on the short roots and on the long
      roots. So
      \begin{equation*}
        \calL_\eff
        =
        \left\{
        \sum_{i=1}^\ell a_i\epsilon_i
        ~\Big|~
        a_1,\dots,a_\ell\in\Z,~
        \sum_{i=1}^\ell a_i\equiv0\mod 2
        \right\}
      \end{equation*}
      This entails $\calL/\calL_\eff=\Z_2$.
      The second statement of the
      proposition is clear from the construction.
    \end{proof}

    \begin{lem}\label{lem:universeZBil}
      There is a $\calV$-invariant bihomomorphism
      \begin{math}
        (\sbullet|\sbullet):~
        \calL\times\calL\to\Z
      \end{math}
      with the following universal property:
      If $\beta:~\calL\times\calL\to\Z$ is a $\calV$-invariant
      bihomomorphism, then there is a unique group
      homomorphism $\phi:~\Z\to\Z$ such that the following diagram
      commutes:
      \begin{equation*}
        \xymatrix{
          \calL\times\calL
          \ar[rr]^{(\sbullet,\sbullet)}
          \ar[rrd]_{\beta}
          &&
          \Z
          \ar@{.>}[d]^\phi
          \\
          &&
          \Z.
        }
      \end{equation*}
      Moreover, the bihomomorphism $(\sbullet|\sbullet)$ is symmetric and
      satisfies 
      $(\alpha|\beta)=0$ for
      $\alpha,\beta\in\Delta$ with $r_\alpha\perp r_\beta$.
    \end{lem}

    \begin{proof}
      According to
      \cite{bou}~Ch.~{\sc vi},~$\mathrm{n}^{\mathrm{o}}$~1.1~Proposition~3,
      there is a $\calV$-invariant
      symmetric bihomomorphism $(\sbullet|\sbullet):~\calL\times\calL\to\Z$. 
      Since $(\calL|\calL)$
      is a non-trivial subgroup of $\Z$, it is possible to pass to a
      $\calV$-invariant symmetric bihomomorphism
      $(\sbullet,\sbullet):~\calL\times\calL\to\Z$ with
      $(\calL,\calL)=\Z$. It satisfies 
      $(\alpha|\beta)=0$ for
      $\alpha,\beta\in\Delta$ with $r_\alpha\perp r_\beta$.

      Denote the tensor product $\calL\otimes_\Z\Q$ 
      by $\calL_\Q$ and suppose
      \begin{math}
        \beta:~
        \calL\times\calL
        \to
        \Z
      \end{math}
      is a $\calV$-invariant bihomomorphism.
      It can be extended to a $\calV$-invariant
      $\Q$-bilinear form
      $\beta_\Q:~\calL_\Q\times\calL_\Q\to\Q$. In the same way, we
      extend $(\sbullet,\sbullet)$ to $(\sbullet,\sbullet)_\Q$. According to 
      \cite{bou}~Ch.~{\sc v},~$\mathrm{n}^{\mathrm{o}}$~2.2~Proposition~1, the form
      $\beta_\Q$  is
      a $\Q$-multiple of $(\sbullet,\sbullet)_\Q$. Since
      $1\in(\calL,\calL)$ and $\beta(\calL,\calL)\subseteq\Z$,
      we have $\beta=k(\sbullet,\sbullet)$ for some $k\in\Z$. If we set
      \begin{math}
        \phi:~\Z\to\Z,~z\mapsto kz
      \end{math} then the diagram in the lemma commutes.
    \end{proof}

    Set
    \begin{equation*}
      \calL\otimes_\calV\calL
      :=
      \frac{
        \calL\otimes\calL
      }{
        \big\langle (v.\lambda)\otimes(v.\mu)-\lambda\otimes\mu
        ~\big|~
        v\in\calV;~\lambda,\mu\in\calL
        \big\rangle
      }
    \end{equation*}
    and denote by
    \begin{math}
      \otimes_\calV:~
      \calL\times\calL
      \to
      \calL\otimes_\calV\calL
    \end{math}
    the associated $\calV$-invariant bihomomorphism.

    \begin{thm}\label{thm:otimesV}
      We have
      \begin{equation*}
        \calL\otimes_\calV\calL
        \cong
        \begin{cases}
          \Z\times\Z_2
          &
          \text{if $\Delta$ is of type $B_\ell(\ell\ge2)$ 
            or $BC_\ell(\ell\ge2)$}
          \\
          \Z
          &
          \text{otherwise},
        \end{cases}
      \end{equation*}
      the bihomomorphism
      \begin{math}
        \calL^2\to
        \calL\otimes_\calV\calL
      \end{math}
      is symmetric
      and
      \begin{math}
        \calL\otimes_\calV\calL
      \end{math}
      is generated by the set
      \begin{equation*}
        \{\alpha\otimes_\calV\alpha
        ~|~
        \alpha\in\Delta_\sh\}
        \cup
        \{\alpha\otimes_\calV(r_\beta.\alpha)
        ~|~
        \alpha\in\Delta_\sh,\beta\in\Delta\}.
      \end{equation*}
    \end{thm}
  
    Before we begin the proof of this theorem we will establish two
    preliminary results.

    \begin{lem}\label{lem:tensorSym}
      The map $\otimes_\calV$ is symmetric and
      $\alpha\otimes_\calV\beta=\beta\otimes_\calV\alpha=0$ for every pair of
      roots $\alpha$ and $\beta\in B$ with $\alpha\not\sim\beta$.
    \end{lem}

    \begin{proof}
      Let $\alpha$ and $\beta$ be adjacent in $B$. Then 
      Table~\ref{tab:rootSystems} shows that
      \begin{math}
        \langle\check\beta,\alpha\rangle=-1
      \end{math}
      or
      \begin{math}
        \langle\check\alpha,\beta\rangle=-1.
      \end{math}
      Without loss of generality we assume the latter. Then
      \begin{eqnarray*}
        \beta\otimes_\calV(\beta+\alpha)
        &=&
        \beta\otimes_\calV(r_\alpha.\beta)
        ~=~
        (r_\alpha.\beta)\otimes_\calV\beta
        ~=~
        (\beta+\alpha)\otimes_\calV\beta.
      \end{eqnarray*}
      this entails
      $\alpha\otimes_\calV\beta=\beta\otimes_\calV\alpha$. 

      Now suppose $\alpha$ and $\beta\in B$ with $\alpha\not\sim\beta$. 
      This entails
      \begin{math}
        \langle\check\beta,\alpha\rangle
        =
        0.
      \end{math}
      Since the Dynkin graph is connected and all except possibly one
      edge are simple there must be a root
      $\gamma\in B$ such that $\alpha$ and $\gamma$ are simply
      adjacent,  or $\beta$ and $\gamma$ are simply adjacent. 
      We may assume the latter. So we have
      \begin{math}
        \langle\check\beta,\gamma\rangle
        =
        -1.
      \end{math}
      Then
      \begin{eqnarray*}
        \alpha\otimes_\calV\gamma
        &=&
        (r_\beta.\alpha)\otimes_\calV(r_\beta.\gamma)
        ~=~
        \alpha\otimes_\calV
        (\gamma+\beta)
        \\
        \gamma\otimes_\calV\alpha
        &=&
        (r_\beta.\gamma)\otimes_\calV(r_\beta.\alpha)
        ~=~
        (\gamma+\beta)\otimes_\calV\alpha,
      \end{eqnarray*}
      so, $\alpha\otimes_\calV\beta=\beta\otimes_\calV\alpha=0$. 
    \end{proof}
    \begin{lem}\label{lem:rootCombs}
      Let $\alpha$, $\beta$ and $\gamma$ be roots.
      \begin{enumerate}
  \item If
        \begin{math}
          \alpha\not\sim\gamma
        \end{math}
        then
        \begin{math}
          \langle\check\beta,\gamma\rangle\alpha\otimes_\calV\beta
          =
          \langle\check\beta,\alpha\rangle\beta\otimes_\calV\gamma.
        \end{math}
  \item
        \begin{math}
          2\alpha\otimes_\calV\beta
          =
          \langle\check\alpha,\beta\rangle\alpha\otimes_\calV\alpha.
        \end{math}
      \end{enumerate}
    \end{lem}
    \begin{proof}
      Ad (i):
      \begin{math}
        \begin{array}[t]{rcl}
          \langle\check\beta,\gamma\rangle\alpha\otimes_\calV\beta
          &=&
          -\alpha\otimes_\calV(\gamma-\langle\check\beta,\gamma\rangle\beta)
          ~=~
          -\alpha\otimes_\calV(r_\beta.\gamma)
          \\
          &=&
          -(r_\beta.\alpha)\otimes_\calV\gamma
          ~=~
          -(\alpha-\langle\check\beta,\alpha\rangle\beta)\otimes_\calV\gamma
          \\
          &=&
          \langle\check\beta,\alpha\rangle\beta\otimes_\calV\gamma.
        \end{array}
      \end{math}

      Ad (ii):
      \begin{math}
        \begin{array}[t]{rcl}
        \alpha\otimes_\calV\beta
        &=&
        (r_\alpha.\alpha)\otimes_\calV(r_\alpha.\beta)
        ~=~
        -\alpha\otimes_\calV(\beta-\langle\check\alpha,\beta\rangle\alpha).
      \end{array}
    \end{math}
  \end{proof}

    Now we prove the theorem.
    \begin{proof}
      Due to Remark~\ref{rem:nonReduced} we my assume that the root system
      $\Delta$ is reduced.
      Note that with the map $(\sbullet|\sbullet)$ from 
      Lemma~\ref{lem:universeZBil} and the universal property of
      $\otimes_\calV$ we have the following commuting diagram:
      \begin{equation*}
        \xymatrix{
          \calL^2
          \ar[rr]
          \ar[rrd]_{(\sbullet|\sbullet)}
          &&
          \calL\otimes_\calV\calL
          \ar[d]
          \\
          &&
          \Z.
        }
      \end{equation*}
      So in order to prove
      $\calL\otimes_\calV\calL\cong\Z$, it suffices to show that
      $\calL\otimes_\calV\calL$ is cyclic.

      {\bf Case 1:}~~ The root system $\Delta$ is of type $A_1$. Then
      $\calL\cong\Z$, generated by the single root $\alpha\in B$. So
      $\calL\otimes_\calV\calL$ 
      is 
      generated by
      $\alpha\otimes_\calV\alpha$. 

      {\bf Case 2:}~~ The root system $\Delta$ is simply laced 
      and of rank $\ell\ge2$,
      i.e. of one of the types 
      $A_\ell(\ell\ge2)$, 
      $D_\ell(\ell\ge4)$, 
      $E_6$, $E_7$ and $E_8$.
      Let $\alpha$ and $\alpha'$ be adjacent roots in $B$.
      Lemma~\ref{lem:rootCombs}~(ii) implies
      $\alpha\otimes_\calV\alpha=\alpha'\otimes_\calV\alpha'=-2\alpha\otimes_\calV\alpha'$. 
      Moreover, if $\alpha''\in B$ is adjacent to $\alpha'$ then
      $\alpha\not\sim\alpha''$ since the Dynkin diagram contains no loops.
      Lemma~\ref{lem:rootCombs}~(i) implies
      $\alpha\otimes_\calV\alpha'=\alpha'\otimes_\calV\alpha''$. Since the Dynkin
      diagram is connected and in view of Lemma~\ref{lem:tensorSym} we
      obtain the following for $\gamma$ and $\delta\in B$:
      \begin{equation*}
        \gamma\otimes_\calV\delta=
        \begin{cases}
          \alpha\otimes_\calV\alpha'
          &
          \text{if $\gamma$ and $\delta$ are adjacent}
          \\
          -2\alpha\otimes_\calV\alpha'
          &
          \text{if $\gamma=\delta$}
        \end{cases}
      \end{equation*}
      Again in view of
      Lemma~\ref{lem:tensorSym},
      this proves that $\calL\otimes_\calV\calL$ is generated by
      \begin{eqnarray*}
        \alpha'\otimes_\calV(r_\alpha.\alpha')
        &=&
        \alpha'\otimes_\calV(\alpha'+\alpha)
        ~=~
        \alpha'\otimes_\calV\alpha'
        +
        \alpha'\otimes_\calV\alpha
        \\
        &=&
        -2\alpha\otimes_\calV\alpha'
        +
        \alpha'\otimes_\calV\alpha
        ~=~
        -\alpha\otimes_\calV\alpha'.
      \end{eqnarray*}
      
      In the following cases there will be two different root
      lengths. Let $\alpha$ and $\beta$ be the unique pair of adjacent
      roots in
      $B$ such that $\alpha$ is short and $\beta$ is long.

      {\bf Case 3:}~~ There is at least one short root $\alpha'\in B$ other
      than $\alpha$, i.e. the root system $\Delta$ is of type
      $C_\ell(\ell\ge3)$ or $F_4$. Note that we have
      \begin{math}
        \langle\check\alpha,\beta\rangle=-2
      \end{math}
      and 
      \begin{math}
        \langle\check\beta,\alpha\rangle=-1.
      \end{math}      
      If $\gamma$ and
      $\delta\in B$ then
      \begin{equation*}
        \gamma\otimes_\calV\delta=
        \left\{
        \begin{array}{clr}
          2\alpha\otimes_\calV\alpha'
          &
          \text{if $\gamma=\alpha$ and $\delta=\beta$}
          &
          (\text{by Lemma~\ref{lem:rootCombs}~(ii)}),
          \\
          \alpha\otimes_\calV\alpha'
          &
          \text{if $\gamma$ and $\delta\in\Delta_\sh$ and
            $\gamma\sim\delta$}
          &
          (\text{by Lemma~\ref{lem:rootCombs}~(ii)}),
          \\
          2\alpha\otimes_\calV\alpha'
          &
          \text{if $\gamma$ and $\delta\in\Delta_\lng$ and
            $\gamma\sim\delta$}
          &
          (\text{by Lemma~\ref{lem:rootCombs}~(ii)}),
          \\
          -2\alpha\otimes_\calV\alpha'
          &
          \text{if $\gamma=\delta\in\Delta_\sh$}
          &
          (\text{by Lemma~\ref{lem:rootCombs}~(i)}),
          \\
          -4\alpha\otimes_\calV\alpha'
          &
          \text{if $\gamma=\delta\in\Delta_\lng$}
          &
          (\text{by Lemma~\ref{lem:rootCombs}~(i)}).
        \end{array}
      \right.
      \end{equation*}
      As in the previous case $\calL\otimes_\calV\calL$ is generated
      by
      $\alpha'\otimes_\calV(r_\alpha.\alpha')$.

      {\bf Case 4:}~~ The root system $\Delta$ is of type $G_2$. Then we
      obtain
      \begin{equation*}
        \beta\otimes_\calV\beta=-2\alpha\otimes_\calV\beta
        ~~~\text{and}~~~
        3\alpha\otimes_\calV\alpha=-2\alpha\otimes_\calV\beta.
      \end{equation*}
      The group $\calL\otimes_\calV\calL$ is generated by the set
      $\{\alpha\otimes_\calV\beta,\alpha\otimes_\calV\alpha\}$. Thus
      the group $\calL\otimes_\calV\calL$ is a quotient of the
      group
      \begin{math}
        \Z^2/\big(\Z(2,3)\big)\cong\Z.
      \end{math}
      This is a cyclic group. Moreover, since
      \begin{equation*}
        \alpha\otimes_\calV\beta
        ~=~
        \alpha\otimes_\calV(r_\beta.\alpha)
        -
        \alpha\otimes_\calV\alpha
      \end{equation*}
      the group is generated by the set 
      $\{\alpha\otimes_\calV(r_\beta.\alpha),
      \alpha\otimes_\calV\alpha\}$.

      {\bf Case 5:}~~ The root system $\Delta$ is of type
      $B_\ell(\ell\ge2)$. 
      If $\gamma$ and
      $\delta\in B$ then
      \begin{equation*}
        \gamma\otimes_\calV\delta=
        \begin{cases}
          \alpha\otimes_\calV\beta
          &
          \text{if $\gamma$ and $\delta$ are adjacent and
            both of them are long}
          \\
          -2\alpha\otimes_\calV\beta
          &
          \text{if $\gamma=\delta$ and $\gamma$ is long}.
        \end{cases}
      \end{equation*}
      The group $\calL\otimes_\calV\calL$ 
      is generated by 
      $\{\alpha\otimes_\calV\beta,\alpha\otimes_\calV\alpha\}$ and satisfies the
      relation $2\alpha\otimes_\calV\beta=-2\alpha\otimes_\calV\alpha$. 
      As in the previous case, the set
      $\{\alpha\otimes_\calV(r_\beta.\alpha),
      \alpha\otimes_\calV\alpha\}$ is a generating set.
      But this time, the group $\calL\otimes_\calV\calL$ 
      is a quotient of the group
      \begin{math}
        \Z^2/\big(\Z(2,2)\big)\cong\Z\times\Z_2.
      \end{math}
      To show that $\calL\otimes_\calV\calL$ is actually
      isomorphic to this group, we will construct a surjective
      $\calV$-invariant bihomomorphism
      $\calL\times\calL\to\Z\times\Z_2$. Note that the considerations
      above imply $-(\alpha|\beta)=(\alpha|\alpha)=1$.

      Due to
      Proposition~\ref{prop:calLEff} we have
      $\calL/\calL_\eff\cong\Z_2$. Denote the quotient
      homomorphism 
      by $\calL\to\Z_2,~\lambda\mapsto\overline\lambda$. Set
      \begin{equation*}
        b:~
        \calL^2\to\Z\times\Z_2,~
        (\lambda,\mu)
        \mapsto
        \big((\lambda|\mu),\overline\lambda\overline\mu\big).
      \end{equation*}
      This is a $\calV$-invariant bihomomorphism and it is surjective
      due to
      \begin{lastdisplay}
        \begin{align*}
          -b(\alpha,\beta)=(1,0)
          ~~~\text{and}~~~
          b(\alpha,\alpha+\beta)=(0,1).
        \end{align*}
      \end{lastdisplay}
    \end{proof}

    Set
    \begin{equation*}
      \calL\boxtimes\calL
      :=
      \frac{
        \calL\otimes_\calV\calL
      }{
        \big\langle \alpha\otimes_\calV\beta
        ~\big|~
        \alpha,\beta\in\Delta_\sh\text{~with~}r_\alpha\perp r_\beta
        \big\rangle
      }
    \end{equation*}
    and denote by
    \begin{math}
      \boxtimes:~
      \calL\times\calL
      \to
      \calL\boxtimes\calL
    \end{math}
    the associated $\calV$-invariant bihomomorphism.

    \begin{cor}\label{cor:boxtimes}
      We have
      \begin{math}
        \calL\boxtimes\calL
        \cong
        \Z
      \end{math}
      and
      \begin{math}
        \calL\boxtimes\calL
      \end{math}
      is generated by the set
      \begin{equation*}
        \{\alpha\boxtimes\alpha
        ~|~
        \alpha\in\Delta_\sh\}
        \cup
        \{\alpha\boxtimes(r_\beta.\alpha)
        ~|~
        \alpha\in\Delta_\sh,\beta\in\Delta\}.
      \end{equation*}
      Moreover $\alpha\boxtimes\beta=0$ for
      $\alpha,\beta\in\Delta$ with $r_\alpha\perp r_\beta$.      
    \end{cor}

    \begin{proof}
      According to Lemma~\ref{lem:universeZBil} 
      we have $(\alpha|\beta)=0$ for
      $\alpha,\beta\in\Delta$ with $r_\alpha\perp r_\beta$. So there is
      a homomorphism $\calL\boxtimes\calL\to\Z$ such that
      the following diagram is commutative:
      \begin{equation*}
        \xymatrix{
          &&
          \calL\otimes_\calV\calL
          \ar[d]
          \\
          \calL^2
          \ar[rru]
          \ar[rr]
          \ar[rrd]_{(\sbullet|\sbullet)}
          &&
          \calL\boxtimes\calL
          \ar[d]
          \\
          &&
          \Z
        }
      \end{equation*}
      We have to revisit
      Case~5 in the proof of the previous theorem. 
      Consider the root
      $\omega:=r_\beta.\alpha=\alpha+\beta$. Note that
      \begin{equation*}
        \langle\check\alpha,\omega\rangle
        ~=~
        \langle\check\alpha,\alpha\rangle
        +
        \langle\check\alpha,\beta\rangle
        ~=~
        0
      \end{equation*}
      This entails 
      $\langle\check\omega,\alpha\rangle=0$, so
      $[r_\alpha,r_{\omega}]=1$. This entails the relation
      \begin{equation*}
        0
        ~=~
        \alpha\boxtimes\omega
        ~=~
        \alpha\boxtimes\alpha
        +
        \alpha\boxtimes\beta.
      \end{equation*}
      This means that
      $\calL\boxtimes\calL$ is cyclic.
    \end{proof}

    Set
    \begin{equation*}
      \calL\otimes_\calV\check\calL
      :=
      \frac{
        \calL\otimes\check\calL
      }{
        \big\langle (v.\lambda)\otimes(v.\mu)-\lambda\otimes\mu
        ~\big|~
        v\in\calV,~\lambda\in\calL,~\mu\in\check\calL
        \big\rangle
      }
    \end{equation*}
    and denote by
    \begin{math}
      \otimes_\calV:~
      \calL\times\check\calL
      \to
      \calL\otimes_\calV\check\calL
    \end{math}
    the associated $\calV$-invariant bihomomorphism.

    \begin{thm}\label{thm:otimesVCheck}
      We have
      \begin{equation*}
        \calL\otimes_\calV\check\calL
        \cong
        \begin{cases}
          \Z\times\Z_2
          &
          \text{if $\Delta$ is of type $BC_\ell(\ell\ge2)$}
          \\
          \Z
          &
          \text{otherwise}.
        \end{cases}
      \end{equation*}
    \end{thm}

    We will start with versions of Lemma~\ref{lem:rootCombs} and
    Lemma~\ref{lem:tensorSym} appropriate
    for elements of $\calL\otimes_\calV\check\calL$.
    \begin{lem}
      Suppose $\alpha$ and $\beta\in B$. Then the following
      are satisfied:
      \begin{enumerate}
  \item If $\alpha$ and $\beta$ have the same length then
        $\alpha\otimes_\calV\check\beta
        =
        \beta\otimes_\calV\check\alpha.$ 
  \item If $\alpha\not\sim\beta$ then 
        $\alpha\otimes_\calV\check\beta
        =
        \beta\otimes_\calV\check\alpha
        =
        0$.
      \end{enumerate}
    \end{lem}
    \begin{proof}
      Let $\alpha$ and $\beta$ be adjacent in $B$ and of the same
      length. Then 
      Table~\ref{tab:rootSystems} shows that
      \begin{math}
        \langle\check\beta,\alpha\rangle=
        \langle\check\alpha,\beta\rangle=-1.
      \end{math}
      So,
      \begin{eqnarray*}
        \beta\otimes_\calV(\check\beta+\check\alpha)
        &=&
        \beta\otimes_\calV(r_\alpha.\check\beta)
        ~=~
        (r_\alpha.\beta)\otimes_\calV\check\beta
        ~=~
        (\beta+\alpha)\otimes_\calV\check\beta.
      \end{eqnarray*}
      This entails
      $\alpha\otimes_\calV\check\beta=\beta\otimes_\calV\check\alpha$. 

      Now suppose $\alpha$ and $\beta\in B$ with $\alpha\not\sim\beta$. 
      This entails
      \begin{math}
        \langle\check\alpha,\beta\rangle
        =
        \langle\check\beta,\alpha\rangle
        =
        0.
      \end{math}
      Since the Dynkin graph is connected and all except possibly one
      edge are simple,
      there must be a root
      $\gamma\in B$ such that $\alpha$ and $\gamma$ are simply
      adjacent,  or $\beta$ and $\gamma$ are simply adjacent. 
      We may assume the latter. So we have
      \begin{math}
         \langle\check\gamma,\beta\rangle
         =
        -1.
      \end{math}
      Then
      \begin{eqnarray*}
        \alpha\otimes_\calV\check\gamma
        &=&
        (r_\beta.\alpha)\otimes_\calV(r_\beta.\check\gamma)
        ~=~
        \alpha\otimes_\calV
        (\check\gamma+\check\beta)
        ~~~\text{and}
        \\
        \gamma\otimes_\calV\check\alpha
        &=&
        (r_\beta.\gamma)\otimes_\calV(r_\beta.\check\alpha)
        ~=~
        (\gamma+\beta)\otimes_\calV\check\alpha,
      \end{eqnarray*}
      so, $\alpha\otimes_\calV\check\beta=\beta\otimes_\calV\check\alpha=0$. 
    \end{proof}

    \begin{lem}\label{lem:rootCombsCheck}
      Let $\alpha$, $\beta$ and $\gamma$ be roots.
      \begin{enumerate}
  \item If 
        \begin{math}
          \alpha\not\sim\gamma
        \end{math}
        then
        \begin{math}
          \langle\check\gamma,\beta\rangle\alpha\otimes_\calV\check\beta
          =
          \langle\check\beta,\alpha\rangle\beta\otimes_\calV\check\gamma.
        \end{math}
  \item
        \begin{math}
          2\alpha\otimes_\calV\check\beta
          =
          \langle\check\beta,\alpha\rangle\alpha\otimes_\calV\check\alpha
          =
          \langle\check\beta,\alpha\rangle\beta\otimes_\calV\check\beta.
        \end{math}
      \end{enumerate}
    \end{lem}
    \begin{proof}
      Ad (i):
      \par\noindent
      \begin{math}
        \begin{array}[t]{rcl}
          \langle\check\gamma,\beta\rangle\alpha\otimes_\calV\check\beta
          &=&
          -\alpha\otimes_\calV
          (\check\gamma-\langle\check\gamma,\beta\rangle\check\beta)
          ~=~
          -\alpha\otimes_\calV(r_\beta.\check\gamma)
          \\
          &=&
          -(r_\beta.\alpha)\otimes_\calV\check\gamma
          ~=~
          -(\alpha-\langle\check\beta,\alpha\rangle\beta)\otimes_\calV\check\gamma
          \\
          &=&
          \langle\check\beta,\alpha\rangle\beta\otimes_\calV\check\gamma.
        \end{array}
      \end{math}
      \par\noindent
      \begin{tabular}[b]{c}
        Ad (ii):\\
        ~
      \end{tabular}
      \begin{math}
        \begin{array}[b]{rcl}
          \alpha\otimes_\calV\check\beta
          &=&
          (r_\alpha.\alpha)\otimes_\calV(r_\alpha.\check\beta)
          ~=~
          -\alpha\otimes_\calV
          (\check\beta-\langle\check\beta,\alpha\rangle\check\alpha),
          \\
          \alpha\otimes_\calV\check\beta
          &=&
          (r_\beta.\alpha)\otimes_\calV(r_\beta.\check\beta)
          ~=~
          (\alpha-\langle\check\beta,\alpha\rangle\beta)\otimes_\calV
          (-\check\beta).
        \end{array}
      \end{math}
      \vspace{1cm}
    \end{proof}
    
    Now we present the proof of the theorem.

    \begin{proof}
      If the root system is reduced and simply laced, or if it is
      non-reduced, then we can identify the $\calV$-modules 
      $\calL$ and $\check\calL$ by Remark~\ref{rem:calLId}. In those
      cases, we are done by the previous theorem. 
      So we focus on the non-simply laced reduced case:

      We use the identification $\calL\subseteq\check\calL$ from
      Lemma~\ref{lem:nonsimplyId}. 
      Due to the universal property of
      $\calL\otimes_\calV\check\calL$ 
      we have the following commutative diagram:
      \begin{equation*}
        \xymatrix{
          \calL\times\check\calL
          \ar[rr]
          \ar[rrd]
          &&
          \calL\otimes_\calV\check\calL
          \ar[d]
          \\
          &&
          \check\calL\otimes_\calV\check\calL.
        }
      \end{equation*}
      Since $k_\Delta\check\calL\subseteq\calL$ the bihomomorphism 
      \begin{math}
        \calL\times\check\calL
        \to
        \check\calL\otimes_\calV\check\calL
      \end{math}
      in this diagram
      has
      \begin{math}
        k_\Delta\check\calL\otimes_\calV\check\calL
      \end{math}
      and thus 
      a copy of $\Z$ in its image. So for our proof it suffices
      to show that 
      \begin{math}
        \calL\otimes_\calV\check\calL
      \end{math}
      is cyclic.

      Let $\alpha$ and $\beta$ be the unique pair of adjacent
      roots in
      $B$ such that $\alpha$ is short and $\beta$ is long and recall
      that 
      \begin{math}
        k=-\langle\check\alpha,\beta\rangle.
      \end{math}
      Then 
      \begin{eqnarray*}
        \beta\otimes_\calV(\check\beta+\check\alpha)
        &=&
        \beta\otimes_\calV(r_\alpha.\check\beta)
        ~=~
        (r_\alpha.\beta)\otimes_\calV\check\beta
        ~=~
        (\beta+k_\Delta\alpha)\otimes_\calV\check\beta,
        \text{~~so}
        \\
        \beta\otimes_\calV\check\alpha
        &=&
        k_\Delta\alpha\otimes_\calV\check\beta.
      \end{eqnarray*}
      Moreover, by Lemma~\ref{lem:rootCombsCheck}~(ii), we have
      \begin{align*}
        \alpha\otimes_\calV\check\alpha
        &=
        \beta\otimes_\calV\check\beta
        =
        -2\alpha\otimes_\calV\check\beta.
      \end{align*}

      If there is a root $\alpha'\in B$ with
      $\alpha'\sim\alpha$, 
      i.e. if the root system $\Delta$ is of type
      $C_\ell(\ell\ge3)$ or $F_4$, then $\alpha'$ is short, we have
      $\alpha'\not\sim\beta$  and  
      \begin{align*}
        \alpha'\otimes_\calV\check\alpha
        &=
        \alpha\otimes_\calV\check\beta
      \end{align*}
      by Lemma~\ref{lem:rootCombsCheck}~(i).
      
      If there is a root $\beta'\in B$ with
      $\beta'\sim\beta$, 
      i.e. if the root system $\Delta$ is of type
      $B_\ell(\ell\ge3)$ or $F_4$, then $\beta'$ is long, we have
      $\alpha\not\sim\beta'$  and
      \begin{align*}
        \beta\otimes_\calV\check{(\beta')}
        &=
        \alpha\otimes_\calV\check\beta
      \end{align*}
      by Lemma~\ref{lem:rootCombsCheck}~(i).

      So
      if $\gamma$ and
      $\delta\in B$ then
      \begin{equation*}
        \gamma\otimes_\calV\delta=
        \begin{cases}
          k_\Delta\alpha\otimes_\calV\check\beta
          &
          \text{if $\gamma=\beta$ and $\delta=\alpha$},
          \\
          \alpha\otimes_\calV\check\beta
          &
          \text{if $\gamma,\delta\in\Delta_\sh$ and
            $\gamma\sim\delta$}
          ~~(\text{due to Lemma~\ref{lem:rootCombs}~(i)}),
          \\
          \alpha\otimes_\calV\check\beta
          &
          \text{if $\gamma,\delta\in\Delta_\lng$ and
            $\gamma\sim\delta$}
          ~~(\text{due to Lemma~\ref{lem:rootCombs}~(i)}),
          \\
          -2\alpha\otimes_\calV\check\beta
          &
          \text{if $\gamma=\delta$}.
        \end{cases}
      \end{equation*}
      This shows that $\calL\otimes_\calV\check\calL$ is generated by
      $\alpha\otimes_\calV\check\beta$ and is thus cyclic.
    \end{proof}

    Set
    \begin{equation*}
      \calL\boxtimes\check\calL
      :=
      \frac{
        \calL\otimes_\calV\check\calL
      }{
        \big\langle \alpha\otimes_\calV\check\beta
        ~\big|~
        \alpha,\beta\in\Delta\text{~with~}r_\alpha\perp r_\beta
        \big\rangle
      }
    \end{equation*}
    and denote by
    \begin{math}
      \boxtimes:~
      \calL\times\check\calL
      \to
      \calL\boxtimes\check\calL
    \end{math}
    the associated bihomomorphism.

    \begin{rem}
      Suppose $\Delta$ is non-simply laced and reduced.
      With the identification $\calL\subseteq\check\calL$ from
      Lemma~\ref{lem:nonsimplyId} we have the
      following commutative digram:
      \begin{align}\label{eq:inclusions}
        \xymatrix{
          \calL\times\calL
          \ar[rr]
          \ar[d]
          &&
          \calL\boxtimes\calL
          \ar[d]^\phi
          \\
          \calL\times\check\calL
          \ar[rr]
          \ar[d]
          &&
          \calL\boxtimes\check\calL
          \ar[d]^\psi
          \\
          \check\calL\times\check\calL
          \ar[rr]
          &&
          \check\calL\boxtimes\check\calL.
        }
      \end{align}
      The homomorphism
      $\phi$
      exists due to the universal property of
      $\calL\boxtimes\calL$ and  
      the homomorphism
      $\psi$
      exists due to the universal property of
      $\calL\boxtimes\check\calL$.
    \end{rem}
    
    \begin{cor}
      \begin{math}
        \calL\boxtimes\check\calL\cong\Z.
      \end{math}
    \end{cor}

    \begin{proof}
      In the simply laced reduced case, and in the non-reduced case 
      this follows from Corollary~\ref{cor:boxtimes} with the
      identification from Remark~\ref{rem:calLId}. In the non-simply
      laced reduced case, we know that
      $\calL\boxtimes\check\calL$ is isomorphic to a quotient
      of $\Z$ by Theorem~\ref{thm:otimesVCheck}. So it suffices to see
      that it is not trivial. 
      Consider the homomorphism $\psi$ in (\ref{eq:inclusions}). Since
      $k_\Delta\check\calL\subseteq\calL$, its image contains
      $k_\Delta\check\calL\boxtimes\check\calL$, which is not trivial by
      Corollary~\ref{cor:boxtimes}. 
    \end{proof}

    \begin{cor}\label{cor:inclusions}
      The maps $\phi$ and $\psi$ in (\ref{eq:inclusions}) are injective.
    \end{cor}

    \begin{proof}
      Since each of the groups
      $\calL\boxtimes\calL$,
      $\calL\boxtimes\check\calL$ and
      $\check\calL\boxtimes\check\calL$ is isomorphic to $\Z$, this
      follows from the fact that neither $\phi$ nor $\psi$ is trivial.
    \end{proof}

    \section{Root systems extended by an abelian group}                  %

    In this section we use the results derived in the previous one
    about a finite root system $\Delta$ to understand a root system $R$ of the
    kind $\Delta$ extended by an abelian group $G$. 
    (Compare \cite{yoshii-rootabelian}.) 
    We associate a symmetric system $T$ to $R$ and a particular
    $T$-reflection group $\calW$, the Weyl group of $R$. This notion of a
    Weyl group generalizes the notion of extended affine Weyl groups
    (EAWeGs). We prove the main result
    of this article (Theorem~\ref{thm:ab}), which characterizes the
    relationship between the initial $T$-reflection group $\calU$ and
    the Weyl group $\calW$ in terms of their abelianizations.

    We continue to work with the notation of the previous section, so 
    let $\Delta$ be an irreducible finite root system with
    coroot system $\check\Delta$. 
    Denote its Weyl group by $\calV$, its root lattice by
    $\calL$ and its coroot lattice by $\check\calL$. 
    Let $G$ be an abelian group.

    For $(g,\alpha)\in G\times\Delta$ set
    \begin{equation*}
      r_{(g,\alpha)}:~
      G\times\calL
      \to      
      G\times\calL,~
      (h,\beta)\mapsto
      (h-\langle\check\alpha,\beta\rangle g,
      r_\alpha.\beta).
    \end{equation*}
    Then $r_{(g,\alpha)}$ is a group homomorphism. Moreover, we have 
    \begin{eqnarray}
      (r_{(g,\alpha)}\circ r_{(g,\alpha)})(h,\beta)
      &=&
      r_{(g,\alpha)}
      (h-\langle\check\alpha,\beta\rangle g,
      r_\alpha.\beta)
      \nonumber\\
      &=&
      (
      h
      -\langle\check\alpha,\beta\rangle g
      -\langle\check\alpha,r_\alpha.\beta\rangle g,
      r_\alpha^2.\beta
      )
      \nonumber\\
      &=&
      \big(
      h
      -\langle\check\alpha,\beta\rangle g
      -\langle\check\alpha,\beta\rangle
      (1-\langle\check\alpha,\alpha\rangle)g,
      \beta
      \big)
      \nonumber\\\label{eq:invComp}
      &=&
      (
      h,
      \beta
      )
    \end{eqnarray}
    for $(h,\beta)\in G\times\calL$.
    In particular $r_{(\alpha,g)}$ is an element of the group
    \begin{align*}
      \Aut_G(G\times\calL)=
      \{\phi\in\Aut(G\times\calL)
      ~|~
      \phi(g,0)=(g,0)\text{~for all $g\in G$}
      \}.
    \end{align*}

    \begin{defn}\label{defn:r}
      We shall call a subset $R\subseteq G\times\Delta$ a
      \emph{root system of the kind $\Delta$ extended by $G$}
      if the following are satisfied:
      \medskip
      \begin{enumerate}
  \item[(R0)]
        The projection $R\to\Delta$ is surjective.
  \item[(R1)]
        The image of the projection $R\to G$ generates $G$.
  \item[(R2)] $\{0\}\times\Delta^\rd\subseteq R$.
  \item[(R3)]
        \begin{math}
          r_{\alpha}.R\subseteq R
        \end{math}
        for every $\alpha\in R$.
      \end{enumerate}
      The \emph{type} of the root system $R$ is the type of $\Delta$.
    \end{defn}
    
    \begin{rem}\label{rem:altAxioms}
      If we set
      \begin{equation*}
        S_\alpha:=\{g\in G~|~(g,\alpha)\in R\}
      \end{equation*}
      then each of the conditions (R0) to (R3) as equivalent to each
      of the following conditions, respectively:
      \begin{enumerate}
  \item[(R0')] For every $\alpha\in\Delta$ the set $S_\alpha$ is non-empty.
  \item[(R1')]
        $\bigcup_{\alpha\in\Delta}S_\alpha$ generates $G$.
  \item[(R2')]
        $0\in S_\alpha$ for every $\alpha\in\Delta^\rd$.
  \item[(R3')]
        \begin{math}
          S_\beta-\langle\check\alpha,\beta\rangle S_\alpha
          \subseteq
          S_{r_\alpha.\beta}
          ~~~\text{for all $\alpha$ and $\beta\in\Delta$.}
        \end{math}
      \end{enumerate}
      These are the axioms for a
      \emph{(not necessarily reduced) root system of type $\Delta$ 
        extended by $G$} defined in
      \cite{yoshii-rootabelian}. So our definition coincides with the
      one given there. 
      The anisotropic roots $R^\times$
      of an extended affine root system (EARS)  
      (see, for instance, \cite{EALA_AMS}) form a root system of the
      kind $\Delta$ extended by a finitely generated free abelian
      group $G$. So the results that follow are applicable to EARSs.
    \end{rem}

    We present some results about the sets $S_\alpha$ taken from
    \cite{yoshii-rootabelian}~Section 3. If two roots $\alpha$ and
    $\beta\in\Delta$ have the same length, then
    $S_\alpha=S_\beta$. This allows us to define $S_\sh$, $S_\lng$ and
    $S_\ex$ as follows. Let $x\in\{\sh,\lng,\ex\}$. If $\Delta_x$ is
    not empty then 
    \begin{math}
      S_x:=S_\alpha
    \end{math}
    where $\alpha\in\Delta_x$. Provided the following exist we have
    \begin{equation}\label{eq:shLngEx}
      k_\Delta S_\lng\subseteq S_\ex\subseteq S_\lng,
      ~~~
      k_\Delta S_\sh\subseteq S_\lng\subseteq S_\sh,
      ~~~\text{and}~~~
      k_\Delta^2S_\sh\subseteq S_\ex\subseteq S_\sh,
      ~~~
    \end{equation}
    where $k_\Delta$ is the number defined in (\ref{eq:k}) in the previous section.
    This means that condition (R1') (and thus condition (R1)) 
    is equivalent to $\langle S_\sh\rangle=G$.

    From now on, let $G$ be a torsion-free group and
    let $R$ be a root system of the kind $\Delta$
    extended by $G$.
    \begin{rem}\label{rem:homId}
      The group $\Aut_G(G\times\calL)$ is isomorphic to the semidirect
      product $\Hom(L,G)\rtimes\Aut(\calL)$, where we use the left
      action of $\Aut(\calL)$ on $\Hom(\calL,G)$ given by
      $\phi.\psi=\psi\circ\phi^{-1}$ for $\phi\in\Aut(\calL)$ and
      $\psi\in\Hom(\calL,G)$. 
      The isomorphism can be expressed by the action of
      $\Hom(L,G)\rtimes\Aut(\calL)$ on $G\times\calL$ given by
      \begin{align*}
        (\phi,\psi).(g,\lambda)
        =
        \big(g+\phi\circ\psi(\lambda),\psi(\lambda)\big).
      \end{align*}
      We will identify the two isomorphic groups.

      There is an injective group homomorphism
      \begin{align*}
        \check\calL\to\Hom(\calL,\Z),~
        \lambda\mapsto(\mu\mapsto\langle\lambda,\mu\rangle)
      \end{align*}
      This gives rise to a group homomorphism
      \begin{align*}
        G\otimes\check\calL\to G\otimes\Hom(\calL,\Z)
        \cong\Hom(\calL,G),
      \end{align*}
      which is injective since $G$ is torsion-free. 
      We will identify elements
      of $G\otimes\calL$ with their corresponding images in
      $\Hom(\calL,G)$. With the two identifications discussed we have
      \begin{align*}
        r_{(g,\alpha)}=(g\otimes\check\alpha, r_\alpha)
      \end{align*}
      for every $g\in G$ and $\alpha\in\Delta$.
    \end{rem}

    Set 
    \begin{equation*}
      T:=\{r_\alpha~|~\alpha\in R\}\subseteq\Aut_G(G\times\calL).
    \end{equation*}
    With the identification in Remark~\ref{rem:homId} we have
       \begin{align*}
         r{(g,\alpha)}.
         r{(h,\beta)}
         &=
         r_{(g,\alpha)}
         r_{(h,\beta)}
         r_{(g,\alpha)}
         =
         (g\otimes\check\alpha,r_\alpha)
         (h\otimes\check\beta,r_\beta)
         (g\otimes\check\alpha,r_\alpha)
         \\
         &=
         (g\otimes\check\alpha+h\otimes(r_\alpha.\check\beta),
         r_\alpha r_\beta)
         (g\otimes\check\alpha,r_\alpha)
         \\
         &=
         \big(
         g\otimes\check\alpha
         +
         h\otimes(r_\alpha.\check\beta)
         +
         g\otimes(r_\alpha r_\beta.\check\alpha),
         r_\alpha r_\beta r_\alpha
         \big)
         \\
         &=
         \big(
         h\otimes(r_\alpha.\check\beta)
         +
         g\otimes
         (\check\alpha-r_\alpha r_\beta r_\alpha.\check\alpha),
         r_{r_\alpha.\beta}
         \big)
         \\
         &=
         \big(
         h\otimes(r_\alpha.\check\beta)
         +
         g\otimes
         (\langle
         \check\alpha,r_\alpha.\beta
         \rangle
         r_\alpha.\check\beta),
         r_{r_\alpha.\beta}
         \big)
         \\
         &=
         \big(
         (h-\langle\check\alpha,\beta\rangle g)
         \otimes
         (r_\alpha.\check\beta),
         r_{r_\alpha.\beta}
         \big)
         =
         r_{(h-\langle\check\alpha,\beta\rangle g,r_\alpha.\beta)}
         \\
         &=
         r_{r_{(g,\alpha)}(h,\beta)}.
       \end{align*}

    So, for every $s$ and $t\in T$ the element $s.t:=sts^{-1}$ is in
    $T$ and defines a multiplication 
    $T\times T\to T,~(s,t)\mapsto s.t$ on $T$. As observed in
    Example~\ref{exmp:symSys} this multiplication turns $T$ into a 
    symmetric system.

    \begin{defn}
      We call $T$ 
      \emph{the symmetric system associated to $R$}.
    \end{defn}

    Denote the symmetric system associated to $\Delta$ by 
    $\overline T$.
    Then  
    \begin{align*}
      T\to\overline T,~r_{(g,\alpha)}\mapsto\overline{r_{(g,\alpha)}}
      =r_\alpha
    \end{align*}
    is well-defined and provides a surjective
    system morphism. If $\alpha\in\Delta$ 
    is a root such that $2\alpha\in\Delta$, then
    $r_\alpha=r_{2\alpha}$. So, in view of (R2) the assignment
    \begin{math}
      r_\alpha\mapsto\dot r_\alpha 
      =r_{(0,\alpha)}
    \end{math}
     for every
    $\alpha\in\Delta^\red$ defines a map
    \begin{math}
      \overline T\to T,
    \end{math}
    which is a system section for
    $\overline\sbullet$. 
   
    Let $\calA$ be the subgroup of 
    $\Aut_G(G\times\calL)$ generated by
    $T$. Then $\calA$ acts on $T$ by conjugation. 
    Denote by $\sbullet^\calA:~T\to\calA$ the embedding, then
    $(\calA,\sbullet^\calA)$ is a proper $T$-reflection group. 

    According to Remark~\ref{rem:homId} we may identify $\calA$ with
    the subgroup of $(G\otimes\check\calL)\rtimes\Aut(\calL)$
    generated by $T^\calA$. By definition $\calV$ is the subgroup of
    $\Aut(\calL)$ generated by the reflections $r_\Delta$. Let $\calK$
    be the subgroup of $\Hom(\calL,G)$ generated by
    \begin{align*}
      \{g\otimes\check\alpha~|~(g,\alpha)\in R\}.
    \end{align*}
    Then $\calA=\calK\rtimes\calV$. Moreover the projection
    $\calK\rtimes\calV\to\calV$ is $\overline\sbullet$-compatible, and since
    $\calV$ is the initial $\overline T$-reflection group, we have
    $\overline\calA=\calV$. We have 
    \begin{align}\label{eq:tCalK}
      t^\calK=t^\calA\dot t^\calA=
      (g\otimes\check\alpha,r_\alpha)(0,r_\alpha)
      =
      (g\otimes\check\alpha,1)
    \end{align}
    for all $(g,\alpha)\in R$ and $t=r_{(g,\alpha)}$. So $\calK$ can
    be identified with $\calK^\calA_\calV$ with the notation from
    (\ref{eq:kNotation}).

    Set
    \begin{math}
      \calK_\fix:=
      \{
      k\in\calK~|~v.k=k\text{~for every $v\in\calV$}\}.
    \end{math}

    \begin{lem}\label{lem:kFix}
      $\calK_\fix=\{0\}$
    \end{lem}

    \begin{proof}
      Let $k\in\calK_\fix$. So $k\in\Hom(\calL,G)$ with
      \begin{math}
        k(v.\lambda-\lambda)=0
      \end{math}
      for all $\lambda\in\calL$ and $v\in\calV$. In other words
      \begin{math}
        k(\calL_\eff)=\{0\}.
      \end{math}
      Since $2\calL\subseteq\calL_\eff$ due to 
      Proposition~\ref{prop:calLEff} and since $G$ is torsion-free, we
      can conclude $k=0$.
    \end{proof}

    \begin{lem}\label{lem:trivialCenter}
      If $G$ is non-trivial, 
      then the center of $\calA$ is trivial.
    \end{lem}

    \begin{proof}
      If $k_1$, $k_2\in\calK$ and $v\in\calV$ then
      \begin{align}\label{eq:comm}
        \nonumber
        [(k_1,v),(k_2,1)]
        &=
        (k_1,v)(k_2,1)(-v^{-1}.k_1,v^{-1})(-k_2,1)
        \\
        \nonumber
        &=
        (k_1+v.k_2,v)(-v^{-1}.k_1-v^{-1}.k_2,v^{-1})
        \\
        &=
        (k_1+v_1.k_2-k_1-k_2,1)
        =
        (v_1.k_2-k_2,1).
      \end{align}

      Let $\alpha\in\Delta$. Since $G$ is torsion-free, axiom (R1) and 
      (\ref{eq:shLngEx}) imply that $\langle S_\alpha\rangle$ is not
      just the trivial subgroup of $G$. So there is a non-trivial 
      $g\in G$ such that
      $(g,\alpha)\in R$. Suppose that $(k_1,v)$ is in the
      center of $\calA$. 
      Let
      $(k_2,1)=t^\calK=(g\otimes\check\alpha,1)$ with
      $t=r_{(g,\alpha)}$. 
      Then (\ref{eq:comm}) yields
      \begin{math}
        g\otimes\check\alpha
        -
        g\otimes v.\check\alpha
        =
        g\otimes(\check\alpha-v.\check\alpha)
        =0.
      \end{math}
      Since $G$ is torsion-free, this implies
      $v.\check\alpha=\check\alpha$. Since $\alpha\in\Delta$ was
      arbitrary, we may conclude $v=1$. So any element of the center
      of $\calA$ is of the form $(k_1,1)$.

      Now suppose $(k_2,1)$ is in the center of $\calA$. 
      Then (\ref{eq:comm}) implies $k_2\in\calK_\fix$. According to
      Lemma~\ref{lem:kFix}, we have $k_2=0$.
    \end{proof}

    \begin{cor}\label{cor:aInitial}
      If $G$ is not trivial, then $\calA$ is the terminal
      $T$-reflection group.
    \end{cor}

    \begin{proof}
      Suppose $\calB$ is the terminal $T$-reflection group. Then there is a
      reflection morphism $\calA\to\calB$. This is a central
      extension by Lemma~\ref{lem:centralExt}. 
      So it is injective by Lemma~\ref{lem:trivialCenter}.
    \end{proof}


    \begin{rem}
      Suppose $R$ is of non-reduced type, i.e. of type
      $BC_\ell(\ell\ge1)$. As done in \cite{myWeyl}, we will
      associate a trimmed version of $R$ to it. However, here we will
      accomplish this by turning short roots into extra 
      long ones and not the
      other way around, as in the cited article.
      For a root $(g,\alpha)\in R$ we set
      \begin{eqnarray*}
        \trim(g,\alpha)
        &=&
        \begin{cases}
          (2g,2\alpha)&\text{if $\alpha\in\Delta_\sh$,}\\
          (g,\alpha)&\text{otherwise.}
        \end{cases}
      \end{eqnarray*}
      Then we define
      \begin{equation*}
        R'
        :=
        \trim(R)
        \subseteq
        R^\lng\cup R^\ex,
        ~~~
        \Delta'
        :=
        \Delta^\lng\cup\Delta^\ex
        \text{~~~and~~~}
        G'
        :=
        \langle S_\lng\rangle\subseteq G.
      \end{equation*}
      Then $\Delta'$ is an irreducible finite root system. Its
      associated symmetric system can be identified with the symmetric
      system $\overline T$ associated to 
      $\Delta$. The Weyl group $\calV'$ of $\Delta'$ then coincides,
      as a $\overline T$-reflection group, with $\calV$. Since we have
      \begin{align*}
        r_{\trim(\alpha)}=r_\alpha
        \text{~~~and~~~}
        r_\alpha.\trim(\beta)=\trim(r_\alpha.\beta)
      \end{align*}
      for all $\alpha$ and $\beta\in R$, the set $R'$ is a root system
      of kind $\Delta'$ extended by $G'$.

      Note that we have 
      $2(G\times\calL)
      \subseteq
      G'\times\calL'
      \subseteq 
      G\times\calL$. Since $G\times\calL$ is torsion-free,
      every element 
      $\Aut(G'\times\calL')$ extends to a unique element of 
      $\Aut(G\times\calL)$. This provides an injective group
      homomorphism 
      $\Aut(G'\times\calL')\to\Aut(G\times\calL)$. 
      This homomorphism maps
      \begin{math}
        r'_{\trim(\alpha)}
      \end{math}
      to $r_\alpha$ for every $\alpha\in R$. So it restricts to an
      isomorphism $\calA'\to\calA$. Using this isomorphism, we may
      identify the symmetric systems $T'$ and $T$.

      Since our further investigations only depend on the
      symmetric system $T$ with its terminal reflection group $\calA$,
      we may restrict ourselves to the case of a reduced root system
      $\Delta$.  
    \end{rem}

    From now on, let $\Delta$ be a \emph{reduced}
    irreducible root system.
    If $R$ is of simply laced type, i.e. only contains short roots,
    then $\calK=G\otimes\check\calL$ since $S_\sh$ generates
    $G$. If $R$ is of non-simply laced type, we will use the concept
    of a twist decomposition to understand $\calK$ better.

    \begin{defn}\propname{Twist decomposition}
      Suppose $\Delta$ is of non-simply laced type.
      A decomposition $G=G_1\oplus G_2$ as a direct sum is called a
      \emph{twist decomposition (for $R$)}
      if the following conditions are satisfied:
      \begin{enumerate}
  \item[(T1)]
        \begin{math}
          S_\sh
          =
          (S_\sh\cap G_1)+\langle S_\lng\cap G_2\rangle
        \end{math}
  \item[(T2)]
        \begin{math}
          S_\lng
          =
          k_\Delta\langle S_\sh\cap G_1\rangle+(S_\lng\cap G_2)
        \end{math}
      \end{enumerate}
      A root system $R$ will be called 
      \emph{tame} if it is of non-simply-laced type and $G$ has a
      twist decomposition, or if it is of simply laced type.
    \end{defn}

    If $G$ is finitely generated and free then $R$ is tame.
    (See
    \cite{EAWG}~Lemma~1.19). The twist number defined in
    \cite{EAWG}~Definition~1.18 in the case of EARSs is
    the rank of $G_1$ for a corresponding twist decomposition
    $G=G_1\oplus G_2$. 

    \begin{rem}\label{rem:twist}
      In view of the fact that
      \begin{math}
        \langle S_\sh\rangle=G,
      \end{math}
      a decomposition $G=G_1\oplus G_2$ is a twist decomposition if
      and only if the following conditions are satisfied:
      \begin{enumerate}
  \item
        \begin{math}
          S_\sh
          =
          (S_\sh\cap G_1)+G_2,
        \end{math}
  \item
        \begin{math}
          S_\lng
          =
          k_\Delta G_1+(S_\lng\cap G_2),
        \end{math}
  \item
        \begin{math}
          \langle G_2\cap S_\lng\rangle=G_2.
        \end{math}
      \end{enumerate}
      If it is a twist decomposition, then
      \begin{lastenumerate}
        \setcounter{enumi}{3}
  \item
        \begin{math}
          \langle G_1\cap S_\sh\rangle=G_1,
        \end{math}
  \item
        \begin{math}
          \langle S_\sh\rangle=G_1+G_2
        \end{math}
        and
        \begin{math}
          \langle S_\lng\rangle=k_\Delta G_1+G_2.
        \end{math}
      \end{lastenumerate}
    \end{rem}
    
    From now on suppose that $R$ is tame. 
    If we have a twist decomposition $G=G_1\oplus G_2$ then 
    we will use the identification of
    $\calL\subseteq\check\calL$ from Lemma~\ref{lem:nonsimplyId}. In
    this case we have
    \begin{lem}
        \begin{math}
          \calK
          =
          (G_1\otimes\calL)
          \oplus
          (G_2\otimes\check\calL).
        \end{math}
    \end{lem}
    \begin{proof}
      Recall (\ref{eq:tCalK}) and the fact
      that $\calK$ is the subgroup of
      \begin{align*}
        G\otimes\check\calL
        =
        (G_1\otimes\check\calL)
        \oplus
        (G_2\otimes\check\calL)
      \end{align*}
      generated by $T^\calK$. Let 
      $T_\sh^\calK$ be the subset of $\calK$
      consisting of the elements $r_\alpha^\calK$ 
      with $\alpha\in\Delta_\sh$ 
      and let
      $T_\lng^\calK$ be the subset of $\calK$
      consisting of the elements $r_\alpha^\calK$ 
      with $\alpha\in\Delta_\lng$.
      Then
      \begin{align*}
        \langle T_\sh^\calK\rangle
        &=
        (G_1\otimes\calL)
        +
        (G_2\otimes\calL)
        &
        &\text{and}
        &
        \langle T_\lng^\calK\rangle
        &=
        (k_\Delta G_1\otimes\check\calL)
        +
        (G_2\otimes\check\calL).
      \end{align*}
      Since $k\check\calL\subseteq\calL\subseteq\check\calL$ by
      Lemma~\ref{lem:nonsimplyId} we obtain
      the claim of the Lemma.
    \end{proof}

    \begin{defn}
      A reflection bihomomorphism $b:~\calK^2\to\calZ$ 
      is called 
      \emph{admissible} if the following condition
      is satisfied:
      If $s$ and $t\in T$ with $\overline s\perp\overline t$ then
      \begin{math}
        {b\big(s^\calK,t^\calK\big)
        =
        0.}
      \end{math}
    \end{defn}

    \begin{rem}\label{rem:univAdmis}
      There is an abelian group 
      $
      \calK
      \curlywedge
      \calK
      $
      and a admissible reflection bihomomorphism
      \begin{equation*}
      \curlywedge:
      \calK^2
      \to
      \calK
      \curlywedge
      \calK,~
      (x,y)\mapsto
      x
      \curlywedge
      y
      \end{equation*}
      such that the following universal property is satisfied:
      For any admissible
      \begin{math}
        b:\calK^2\to B
      \end{math}
      there is a unique group homomorphism 
      $\phi:\calK\curlywedge\calK\to B$ such that
      the following diagram commutes:
      \begin{equation*}
        \xymatrix{
        \calK^2
        \ar[rr]
        \ar[rrd]^b
        &&
        \calK\curlywedge\calK
        \ar[d]^\phi
        \\
        &&
        B.
      }
      \end{equation*}
      
      The group
      \begin{math}
        \calK\curlywedge\calK
      \end{math}
      can be constructed as the quotient of
      \begin{math}
        \calK\otimes\calK
      \end{math}
      by the subgroup generated by the union of the sets

      \par\noindent
      \begin{minipage}[b]{0.95\linewidth}
      \begin{align*}
        M_1
        &=
        \{
        t^\calK_1\otimes t_2^\calK
        ~|~
        \text{$t_1$ and $t_2\in T$
          with $t_1\perp t_2$}
        \},
        \\
        M_2
        &=
        \big\{
        t_1^\calK\otimes t_2^\calK
        -
        s^\calV.t_1^\calK\otimes s^\calV.t_2^\calK
        ~|~
        \text{$t_1,t_2\in T$
          and $s\in S$}
        \big\},
        \\
        M_3
        &=
        \{          
        t^\calK\otimes t^\calK
        ~|~
        \text{$t\in T$}
        \},
        \\
        M_4
        &=
        \{
        t_1^\calK\otimes t_2^\calK
        +
        t_2^\calK\otimes t_1^\calK
        ~|~
        \text{$t_1$ and $t_2\in T$}
        \}
        \text{~~~and}
        \\
        M_5
        &=          
        \big\{
        t^\calK\otimes s^\calV.t^\calK
        ~|~
        \text{$t\in T$ and $s\in S$}
        \big\}.
      \end{align*}
      \end{minipage}
    \end{rem}

    Recall that $R$ is supposed to be tame and
    consider the map
    \begin{align*}
      (G\otimes\check\calL)^2
      \to
      (G\wedge G)
      \otimes
      (\check\calL\boxtimes\check\calL),~
      (g,\lambda,h,\mu)\mapsto
      (g\wedge h)\otimes(\lambda\boxtimes\mu).
    \end{align*}
    Its restriction to $\calK^2$ is admissible.
    So, by the universal property of 
    \begin{math}
      \calK\curlywedge\calK
    \end{math}
    there is a group homomorphism
$\psi$
    such that the following diagram commutes:
      \begin{equation*}
        \xymatrix{
          \calK^2
          \ar[rr]
          \ar[drr]
          &&
          \calK
          \curlywedge
          \calK
          \ar[d]^\psi
          \\
          &&
        (G\wedge G)
        \otimes
        (\check\calL\boxtimes\check\calL).
        }        
      \end{equation*}

    \begin{thm} 
      The map $\psi$ is injective and 
      if $R$ is simply laced, it is surjective.
    \end{thm}

    \begin{proof}
      Throughout this proof we will take advantage of the fact that
      the tensor product is commutative and associative and identify 
      $\calK\otimes\calK$ with the subgroup of 
      \begin{math}
        G\otimes G\otimes\check\calL\otimes\check\calL
      \end{math}      
      generated by
      \begin{math}
        \{g\otimes h\otimes\alpha\otimes\beta
        ~|~
        (g,\alpha)\text{~and~}(h,\beta)\in R
        \}.
      \end{math}

      First we assume that $R$ is simply laced, so 
      $\calK\cong G\otimes\check\calL$. We will use the sets
      $M_1$ to $M_5$ defined in Remark~\ref{rem:univAdmis} and restate
      them here.
      Let $\calK\boxtimes\calK$ be the
      quotient of $\calK\otimes\calK$ by the subgroup
      generated by the union of the sets
      \begin{align*}
        M_1
        &=
        \{
        g\otimes h\otimes\check\alpha\otimes\check\beta
        ~|~
        \text{$(g,\alpha),(h,\beta)\in R$
          with $r_\alpha\perp r_\beta$}
        \}
        \text{~~~and}
        \\
        M_2
        &=
        \big\{
        g\otimes h\otimes
        \big(
        \check\alpha\otimes\check\beta
        -
        (r_\gamma.\check\alpha)\otimes(r_\gamma.\check\beta)
        \big)
        ~\big|~
        \text{$(g,\alpha),(h,\beta)\in R$ and $\gamma\in\Delta$}
        \big\}.
      \end{align*}
      By
      \begin{math}
        \pi:~
        \calK\otimes\calK
        \to
        \calK\boxtimes\calK  
      \end{math}
      we denote the quotient morphism.
      Recall the properties of $\check\calL\boxtimes\check\calL$
      given in Corollary~\ref{cor:boxtimes}.
      We have
      \begin{equation*}
        \calK\boxtimes\calK
        \cong
        G\otimes G\otimes(\check\calL\boxtimes\check\calL).
      \end{equation*}
      Now set
      \begin{align*}
        M_3
        &=
        \{          
        g\otimes g\otimes\check\alpha\otimes\check\alpha
        ~|~
        (g,\alpha)\in R
        \},
        \\
        M_4
        &=
        \{
        g\otimes h\otimes\check\alpha\otimes\check\beta
        -
        h\otimes g\otimes\check\beta\otimes\calV\check\alpha
        ~|~
        (g,\alpha),(h,\beta)\in R
        \}
        \\
        M_4'
        &=
        \{
        (g\otimes h-h\otimes g)\otimes\check\alpha\otimes\check\beta
        ~|~
        (g,\alpha),(h,\beta)\in R
        \}
        \text{~~~and}
        \\
        M_5
        &=
        \{
        g\otimes g\otimes\check\alpha\otimes(r_\gamma.\check\alpha)
        ~|~
        \text{$(g,\alpha)\in R$ and $\gamma\in\Delta$}
        \}.
      \end{align*}
      Note that $\pi(M_4)=\pi(M_4')$ since
      $(\check\calL)^2\to\check\calL\boxtimes\check\calL$ is
      symmetric according to Theorem~\ref{thm:otimesV}.
      So
      $\calK\curlywedge\calK$ 
      is the quotient of 
      $\calK\boxtimes\calK$ 
      by its subgroup
      generated by 
      \begin{math}
        \pi(M_3)
        \cup
        \pi(M_4')
        \cup
        \pi(M_5).
      \end{math}
      Since
      $\check\calL\boxtimes\check\calL$ is generated by
      \begin{equation*}
        \{
        \check\alpha\boxtimes\check\alpha
        ~|~
        \alpha\in\Delta
        \}
        \cup
        \{
        \check\alpha\boxtimes(r_\beta.\check\alpha)
        ~|~
        \alpha\text{~and~}\beta\in\Delta
        \}
      \end{equation*}
      by Corollary~\ref{cor:boxtimes}, we have
      \begin{equation*}
        \calK\curlywedge\calK
        \cong
        (G\wedge G)\otimes(\check\calL\boxtimes\check\calL).
      \end{equation*}

      Now we assume that $R$ is not simply laced and that there is a
      twist decomposition
      \begin{math}
        G=G_1\oplus G_2.
      \end{math}
      We have
      \begin{eqnarray*}
        \calK\otimes\calK
        &=&
        \big(
        (G_1\otimes\calL)\oplus(G_2\otimes\check\calL)
        \big)
        \otimes
        \big(
        (G_1\otimes\calL)\oplus(G_2\otimes\check\calL)
        \big)
        \\
        &=&
        \begin{array}[t]{clcl}
          &
          (G_1\otimes G_1\otimes\calL\otimes\calL)
          &\oplus&
          (G_1\otimes G_2\otimes\calL\otimes\check\calL)
          \\
          \oplus&
          (G_2\otimes G_1\otimes\check\calL\otimes\calL)
          &\oplus&
          (G_2\otimes G_2\otimes\check\calL\otimes\check\calL).
        \end{array}
      \end{eqnarray*}
      Set
      \begin{eqnarray*}
        \calK
        \overline\boxtimes
        \calK
        &=&
        \begin{array}[t]{clcl}
          &
          \big
          (G_1\otimes G_1\otimes
          (\calL\boxtimes\calL)
          \big)
          &\oplus&
          \big
          (G_1\otimes G_2\otimes
          (\calL\boxtimes\check\calL)
          \big)
          \\
          \oplus&
          \big
          (G_2\otimes G_1\otimes
          (\check\calL\boxtimes\calL)
          \big)
          &\oplus&
          \big
          (G_2\otimes G_2\otimes
          (\check\calL\boxtimes\check\calL)
          \big).
        \end{array}
      \end{eqnarray*}
      Then
      $
      \calK
      \overline\boxtimes
      \calK
      $ 
      can be obtained as a
      quotient of $\calK\otimes\calK$ by some subgroup
      of the group
      ${\langle M_1\cup M_2\rangle}$.
      Denote the quotient homomorphism by 
      $
      \overline\pi:~
      \calK\otimes\calK
      \to
      \calK
      \overline\boxtimes
      \calK
      $.

      Now set
      \begin{eqnarray*}
        \calK
        \overline\curlywedge
        \calK
        &=&
        \begin{array}[t]{clcl}
          &
          \big(
          (G_1\wedge G_1)\otimes
          (\calL\boxtimes\calL)
          \big)
          &\oplus&
          \big(
          (G_1\otimes G_2)\otimes
          (\calL\boxtimes\check\calL)
          \big)
          \\
          \oplus&
          \big(
          (G_2\wedge G_2)\otimes
          (\check\calL\boxtimes\check\calL)
          \big).
        \end{array}
      \end{eqnarray*}
      Then
      $
      \calK
      \overline\curlywedge
      \calK
      $ 
      can be obtained as a
      quotient of 
      $
      \calK
      \overline\boxtimes
      \calK
      $ 
      by some subgroup of the group
      ${\langle\overline\pi(M_3\cup M_4\cup M_4'\cup M_5)\rangle.}
      $
      For this, note that 
      $\check\calL\boxtimes\check\calL$ is generated by
      \begin{align*}
        \{
        \check\alpha\boxtimes\check\alpha
        ~|~
        \alpha\in\Delta_\lng
        \}
        \cup
        \{
        \check\alpha\boxtimes(r_\beta.\check\alpha)
        ~|~
        \alpha\in\Delta_\lng,\beta\in\Delta
        \}
      \end{align*}
      and with the identification $\calL\subseteq\check\calL$ from
      Lemma~\ref{lem:nonsimplyId} the group
      $\calL\boxtimes\calL$ 
      is generated by
      $
      \{
      \check\alpha\otimes_\calV\check\alpha
      ~|~
      \alpha\in\Delta_\sh
      \}
      \cup
      \{
      \check\alpha\otimes_\calV(r_\beta.\check\alpha)
      ~|~
      \alpha\in\Delta_\sh,\beta\in\Delta
      \}.
      $ 
      
      By Corollary~\ref{cor:inclusions}, 
      we have 
      \begin{math}
        \calL\boxtimes\calL
        \subseteq
        \calL\boxtimes\check\calL
        \subseteq
        \check\calL\boxtimes\check\calL.
      \end{math}
      So
      \begin{math}
        \calK
        \overline\boxtimes
        \calK
      \end{math} 
      is a subgroup of
      \begin{eqnarray*}
        G\wedge G\otimes(\check\calL\boxtimes\check\calL)
        &=&
        \begin{array}[t]{clcl}
          &
          \big(
          (G_1\wedge G_1)\otimes
          (\check\calL\boxtimes\check\calL)
          \big)
          \\
          \oplus&
          \big(
          (G_1\otimes G_2)\otimes
          (\check\calL\boxtimes\check\calL)
          \big)
          \\
          \oplus&
          \big(
          (G_2\wedge G_2)\otimes
          (\check\calL\boxtimes\check\calL)
          \big).
        \end{array}
      \end{eqnarray*}
      By the way we have constructed 
      \begin{math}
        \calK
        \overline\curlywedge
        \calK,
      \end{math}
      there is a surjective homomorphism 
      \begin{math}
        \calK
        \overline\curlywedge
        \calK
        \to
        \calK
        \curlywedge
        \calK
      \end{math}
      and overall we have the following commuting diagram:
      \begin{equation*}
        \xymatrix{
          &&
          \calK
          \overline\curlywedge
          \calK
          \ar[d]
          \ar@/^2pc/[dd]^\phi
          \\
          \calK^2
          \ar[rru]
          \ar[rr]
          \ar[rrd]
          &&
          \calK
          \curlywedge
          \calK
          \ar[d]_\psi
          \\
          &&
          (G\wedge G)
          \otimes
          (\check\calL\boxtimes\check\calL),
        }
      \end{equation*}
      where $\phi$ is injective. 
      This means that the map $\psi$ is injective.
    \end{proof}

    \begin{cor}\label{cor:kKTorsionFree}
      The group $\calK\curlywedge\calK$ is isomorphic to a subgroup of
      $G\wedge G$ and is thus torsion-free.
    \end{cor}
    \begin{proof}
      By the preceding theorem, the group $\calK\curlywedge\calK$ is
      isomorphic to a subgroup of
      \begin{math}
        (G\wedge G)
        \otimes
        (\check\calL\boxtimes\check\calL),
      \end{math}
      which in turn is isomorphic to $G\wedge G$ by
      Corollary~\ref{cor:boxtimes}.
    \end{proof}

    \begin{defn}\propname{Weyl group}
      Suppose $R$ is tame. Then
      the $T$-reflection group $\calW(\calA,\curlywedge)$ is called the
      \emph{Weyl group of $R$}.
    \end{defn}

    This new notion of a Weyl group coincides with the existing notion
    of a Weyl group when $R$ can be interpreted as an EARS as in 
    Remark~\ref{rem:altAxioms}.

    Recall that $c_\calU:\calK^2\to(\calK^\calU)'$ 
    (defined in Remark~\ref{rem:comMap}) is a
    reflection bihomomorphism due to
    Proposition~\ref{prop:commutatorMap}. 

    \begin{lem}
      The reflection bihomomorphism $c_\calU$ is admissible.
    \end{lem}

    \begin{proof}
      Suppose $s$ and $t\in T$. There are $(g,\alpha)$ and 
      $(h,\beta)\in R$ such that 
      $s=r_{(g,\alpha)}$ and $t=r_{(h,\beta)}$.
      Now we suppose 
      $\overline s\perp\overline t$, in other words
      $r_\alpha\perp r_\beta$. This implies
      $
      \langle\check\alpha,\beta\rangle
      =  
      0.
      $
      So we have
      \begin{align*}
        s.t
        =
        r_{(g,\alpha)}.r_{(h,\beta)}
        =
        r_{(h+\langle\check\alpha,\beta\rangle g,r_\alpha.\beta)}
        =
        r_{(h,\beta)}
        =
        t.
      \end{align*}
      This entails
      \begin{align*}
        [s^\calU,t^\calU]
        &=
        [s^\calU,\dot t^\calU]
        =
        [\dot s^\calU,t^\calU]
        =
        [\dot s^\calU,\dot t^\calU]
        =
        1
        \text{~~~and thus}
        \\
          \big[
          s^\calU
          \dot s^\calU,
          t^\calU
          \dot t^\calU
          \big]
          &=
          s^\calU
          \dot s^\calU
          t^\calU
          \dot t^\calU
          \dot s^\calU
          s^\calU
          \dot t^\calU
          t^\calU
          =
          t^\calU
          \dot t^\calU
          s^\calU
          \dot s^\calU
          \dot s^\calU
          s^\calU
          \dot t^\calU
          t^\calU
          =1
        \end{align*}
        in $\calU$. Written additively, in the center of $\calU$ this
        means   
        \begin{math}
          c_\calU\big(
          s^\calK,
          t^\calK
          \big)=0.
        \end{math}
      \end{proof}


      The $T$-reflection morphism $\calU\to\calW$ restricts to a
      surjective group homomorphism 
      $\phi:(\calK^\calU)'\to(\calK^\calW)'$.
      
    \begin{prop}
      The map $\phi$ is an isomorphism.
    \end{prop}

    \begin{proof}
      The way $c_\calU$ and $c_\calW$ are defined in
      Remark~\ref{rem:comMap} the following diagram commutes:
      \begin{equation*}
        \xymatrix{
          &&
          \calK\curlywedge\calK
          \ar[d]
          \ar@/^2pc/[ddd]^\id
          \\
          \calK^2
          \ar[urr]^{\curlywedge}
          \ar[rr]^{c_\calU}
          \ar[drr]^{c_\calW}
          \ar[ddrr]^{\curlywedge}
          &&
          (\calK^\calU)'
          \ar[d]^\phi
          \\
          &&
          (\calK^\calW)'
          \ar[d]
          \\
          &&
          \calK\curlywedge\calK.
        }
      \end{equation*}
      The map
      $(\calK^\calW)'\to\calK\curlywedge\calK$ 
      exists due to Corollary~\ref{cor:comMap}. 
      The map $\calK\curlywedge\calK\to(\calK^\calU)'$ exists by the
      universal property of $\curlywedge$.
      It is surjective, so
      $\phi$ is injective. 
    \end{proof}

    Due to the preceding proposition we have the commutative diagram
    \begin{align*}
      \xymatrix{
        1
        \ar[rr]
        \ar[d]
        &&
        (\calK^\calU)'
        \ar[rr]
        \ar[d]
        &&
        (\calK^\calW)'
        \ar[d]
        \\
        \ker(\calU\to\calW)
        \ar[rr]
        \ar[d]_\psi
        &&
        \calU
        \ar[rr]
        \ar[d]
        &&
        \calW
        \ar[d]
        \\
        \ker(\calU_\ab\to\calW_\ab)
        \ar[rr]
        &&
        \calU_\ab
        \ar[rr]
        &&
        \calW_\ab
      }
    \end{align*}
    and by Lemma~\ref{lem:exact} we obtain
    \begin{cor}\label{cor:abKer}
      The map $\psi$ is an isomorphism.
    \end{cor}

    We are no ready to prove the main result of this article.
    Let $\Delta$ be a reduced finite irreducible root system.
    Let $G$ be a free abelian group. 
    Let $R$ be a tame root
    system of kind $\Delta$ extended by $G$.  
    Let $T$ be the symmetric system associated to $R$. Let $\calU$ be
    the initial $T$-reflection group and let $\calW$ be the Weyl
    group. Using the $\ab$-functor we obtain the following commutative
    diagram:
    \begin{align*}
      \xymatrix{
        \ker(\calU\to\calW)
        \ar[rr]
        \ar[d]
        &&
        \calU
        \ar[rr]
        \ar[d]
        &&
        \calW
        \ar[d]
        \\
        \ker(\calU^\ab\to\calW^\ab)
        \ar[rr]
        &&
        \calU^\ab
        \ar[rr]
        &&
        \calW^\ab,
      }
    \end{align*}
    where 
    $\calU\to\calW$ is a $T$-reflection morphism,
    $\calU^\ab\to\calW^\ab$ is a $T^\ab$-reflection morphism and
    $\calU\to\calU^\ab$ and
    $\calW\to\calW^\ab$ are $\ab$-compatible.
    The groups $\calU^\ab$ and $\calW^\ab$ are just the abelianizations
    of $\calU$ and $\calW$ respectively. 
    
    \begin{thm}\label{thm:ab}
      The map
      \begin{math}
        \ker(\calU\to\calW)
        \to
        \ker(\calU^\ab\to\calW^\ab)
      \end{math}
      is an isomorphism.
    \end{thm}

    \begin{proof}
       The composition  
       \begin{math}
         \ker(\calU\to\calW)
         \to
         \ker(\calU_\ab\to\calW_\ab)
         \to
         \ker(\calU^\ab\to\calW^\ab)
       \end{math}
       is an isomorphism 
       due to Corollary~\ref{cor:abKer} 
       and Proposition~\ref{prop:abKerIso}.
    \end{proof}

    \begin{cor}
      If $\ker(\calU_\ab\to\calW_\ab)\neq\{0\}$, then $\calU$ and
      $\calW$ are not isomorphic (as groups).
    \end{cor}
    
    \begin{proof}
      Suppose $\ker(\calU_\ab\to\calW_\ab)\neq\{0\}$.
      According to Lemma~\ref{lem:center} 
      and Corollary~\ref{cor:aInitial} the center of  
      $\calW$ is isomorphic to a the kernel of $\calW\to\calA$, which
      is isomorphic to 
      a subgroup of $\calK\curlywedge\calK$ and is thus torsion-free
      by Corollary~\ref{cor:kKTorsionFree}.  The center of
      $\calU$ contains $\ker(\calU\to\calW)$, which contains elements
      of order two.
    \end{proof}

    \begin{rem}\label{rem:aAb}
      Since $\calA=\calK\rtimes\calV$, the commutator subgroup 
      of $\calA$ is given
      by
      \begin{align*}
        \calA'=\calK_\eff\rtimes\calV'
        \text{~~with~~}
        \calK_\eff:=
        [\calA,\calK]
        =
        \langle k-v.k
        ~|~k\in\calK,v\in\calV\rangle.
      \end{align*}
      We have
      \begin{align*}
        \calK_\eff
        &=
        G_1\otimes\calL_\eff
        \oplus
        G_2\otimes\check\calL_\eff.
      \end{align*}
      We set ${^\ab\calK}:=\calK/\calK_\eff$ and thus have
      \begin{math}
        \calA^\ab={^\ab\calK}\times\calV^\ab.
      \end{math}
      The group $\calV^\ab$ is the initial $\overline T^\ab$-reflection
      group and thus it is a proper reflection group according to
      Construction~\ref{constr:ab}.

      If $\Delta$ is simply laced, then Proposition~\ref{prop:calLEff}
      implies
      \begin{align*}
        {^\ab\calK}
        &\cong
        G\otimes\frac{\check\calL}{\check\calL_\eff}
        \cong
        \begin{cases}
          G\otimes\Z_2
          &
          \text{if $R$ is of type $A_1$,}
          \\
          \{0\}
          &
          \text{otherwise.}
        \end{cases}
      \end{align*}
      If $\Delta$ is non-simply laced and 
      $G=G_1\oplus G_2$ is a twist decomposition, then $^\ab\calK$ is
      isomorphic to
      \begin{lastalign*}{
        {^\ab\calK}
        \cong
        \left(G_1\otimes\frac{\calL}{\calL_\eff}\right)
        \oplus
        \left(G_2\otimes\frac{\check\calL}{\check\calL_\eff}\right)
        \cong
        \begin{cases}
          G\otimes\Z_2
          &
          \text{if $R$ is of type $B_2$,}
          \\
          G_1\otimes\Z_2
          &
          \text{if $R$ is of type $B_\ell(\ell\ge3)$,}
          \\
          G_2\otimes\Z_2
          &
          \text{if $R$ is of type $C_\ell(\ell\ge3)$,}
          \\
          0
          &
          \text{if $R$ is of type $F_4$ or $G_2$.}
        \end{cases}}
      \end{lastalign*}
    \end{rem}

    \begin{prop}\label{prop:rootOrbit}
      We have
      \begin{align*}
        \calA.(h,\beta)
        =
        \begin{cases}
          (
          h
          +\langle\check\Delta,\beta\rangle\langle S_\sh\rangle
          ,\calV.\beta)
          &
          \text{in the simply laced case,}
          \\
          (
          h
          +\langle\check\Delta_\lng,\beta\rangle\langle S_\sh\rangle
          +\langle\check\Delta_\sh,\beta\rangle\langle S_\lng\rangle
          ,\calV.\beta),
          &
          \text{otherwise.}
        \end{cases}
      \end{align*}
      More precisely,
      \begin{align*}
        \calA.(h,\beta)
        =
        \begin{cases}
          (h+2G,\Delta)
          &
          \text{if $R$ is of type $A_1$,}
          \\
          (h+2G_1+G_2,\Delta_\sh)
          &
          \text{if $R$ is of type $B_2$ 
            and $\beta\in\Delta_\sh$,}
          \\
          (h+2G_1+2G_2,\Delta_\lng)
          &
          \text{if $R$ is of type $B_2$ 
            and $\beta\in\Delta_\lng$,}
          \\
          (h+2G_1+G_2,\Delta_\sh)
          &
          \text{if $R$ is of type $B_\ell(\ell\ge3)$
             and $\beta\in\Delta_\sh$,}
          \\
           (h+2G_1+G_2,\Delta_\lng)
           &
           \text{if $R$ is of type $B_\ell(\ell\ge3)$ 
             and $\beta\in\Delta_\lng$,}
           \\
          (h+G_1+G_2,\Delta_\sh)
          &
          \text{if $R$ is of type $C_\ell(\ell\ge3)$ 
            and $\beta\in\Delta_\sh$,}
          \\
          (h+2G_1+2G_2,\Delta_\lng)
          &
          \text{if $R$ is of type $C_\ell(\ell\ge3)$ 
            and $\beta\in\Delta_\lng$,}
          \\
          (G,\Delta)
          &
          \text{if $R$ is of any other type.}
        \end{cases}
      \end{align*}
    \end{prop}
    \begin{proof}
      The first equation in the proposition can be derived in the
      same way as done in the proof of
      \cite{myWeyl}~Proposition~5.1. What follows can be derived from it using
      $\langle S_\sh\rangle=G_1+G_2$ and
      $\langle S_\lng\rangle=k_\Delta G _1+G_2$ and Table~\ref{tab:pairing},
      the data of which is taken from \cite{EAWG} with one correction.      
      \begin{table}[t]
        \centering
        \begin{tabular}{c|c|c|c|c}
          Type
          &
          $\langle\check\Delta_\sh,\Delta_\sh\rangle$
          &
          $\langle\check\Delta_\sh,\Delta_\lng\rangle$
          &
          $\langle\check\Delta_\lng,\Delta_\sh\rangle$
          &
          $\langle\check\Delta_\lng,\Delta_\lng\rangle$
          \\\hline
          $A_1$
          &
          $\{\pm2\}$
          &
          $\emptyset$
          &
          $\emptyset$
          &
          $\emptyset$
          \\\hline
          $B_2$
          &
          $\{\pm1\}$
          &
          $\{0,\pm2\}$
          &
          $\{0,\pm2\}$
          &
          $\{\pm2\}$
          \\\hline
          $B_\ell(\ell\ge3)$
          &
          $\{0,\pm1\}$
          &
          $\{0,\pm1,\pm2\}$
          &
          $\{0,\pm2\}$
          &
          $\{0,\pm2\}$
          \\\hline
          $C_\ell(\ell\ge3)$
          &
          $\{0,\pm1\}$
          &
          $\{0,\pm2\}$
          &
          $\{0,\pm1,\pm2\}$
          &
          $\{0,\pm2\}$
          \\\hline
          $F_4$
          &
          $\{0,\pm1\}$
          &
          $\{0,\pm1,\pm2\}$
          &
          $\{0,\pm1,\pm2\}$
          &
          $\{0,\pm2\}$
          \\\hline
          $G_2$
          &
          $\{0,\pm1\}$
          &
          $\{0,\pm1,\pm2\}$
          &
          $\{0,\pm1,\pm2\}$
          &
          $\{0,\pm3\}$
          \\\hline
          other
          &
          $\{0,\pm1,\pm2\}$
          &
          $\emptyset$
          &
          $\emptyset$
          &
          $\emptyset$
        \end{tabular}
        \caption{Pairing values for the 
          finite reduced irreducible root systems}
        \label{tab:pairing}
      \end{table}
    \end{proof}
    \begin{rem}
      Under the hypothesis of Theorem~\ref{thm:ab} the word problem
      for the presentation (\ref{eq:preConj}) of $\calU$ has the
      following solution:
      The element $t_1^\calU t_2^\calU\cdots t_n^\calU$ is trivial in
      $\calU$ if and only if the following conditions are satisfied:
      \begin{enumerate}
  \item 
        The element
        $\overline{t_1}^\calW\overline{t_2}^\calW
        \cdots\overline{t_n}^\calW$ 
        is trivial in
        $\calV$.
  \item 
        The element
        $t_1^{\calK^\calW}t_2^{\calK^\calW}\cdots t_n^{\calK^\calW}$ 
        is trivial in
        $\calK^\calW$.
  \item
        The element
        $(t_1^\ab)^{\calU^\ab}(t_2^\ab)^{\calU^\ab}
        \cdots
        (t_n^\ab)^{\calU^\ab}$
        is trivial in $\calU^\ab$.
      \end{enumerate}
      Condition (i) can be decided using the solution for the word
      problem for Coxeter groups (see for
      instance,\cite{hum}). Condition (ii) can be decided by
      evaluating the elements in $\calK^\calW$.
      In view of Construction~\ref{constr:ab},
      condition (iii) can be decided using
      Proposition~\ref{prop:rootOrbit}. 
    \end{rem}

    \begin{thm}
      The $T^\ab$-reflection group $\calA^\ab$ is proper.
    \end{thm}

    \begin{proof}
      Recall that
      \begin{math}
        \calA^\ab
        \cong
        {^\ab\calK}\times\calV^\ab.
      \end{math}
      The map $T^\ab\to\calA^\ab$ does not contain 0
      in its image, since the map $\overline T^\ab\to\calV^\ab$
      does not. In the following we will show that
      the map $T^\ab\to\calA^\ab$ is injective.

      Let 
      $(h_1,\beta_1)$
      and
      $(h_2,\beta_2)\in R$. Set 
      $t_1:=r_{(h_1,\beta_1)}$ and
      $t_2:=r_{(h_2,\beta_2)}$. Now suppose  
      $(t_1^\ab)^{\calA^\ab}=(t_2^\ab)^{\calA^\ab}$. So we have
      \begin{align*}
        (h_1\otimes\check\beta_1,r_{\beta_1})^\ab
        =
        (t_1^\calA)^\ab
        =
        (t_1^\ab)^{\calA^\ab}
        =
        (t_2^\ab)^{\calA^\ab}
        =
        (t_2^\calA)^\ab
        =
        (h_2\otimes\check\beta_2,r_{\beta_2})^\ab.
      \end{align*}
      So $\beta_1$ and $\beta_2$ are in the same $\calV$-orbit, say
      $\beta_1=v.\beta_2$. Moreover we have
      \begin{align*}
        h_1\otimes\check\beta_1
        \equiv
        h_2\otimes\check\beta_2        
        \equiv
        h_2\otimes(v.\check\beta_2)        
        \equiv
        h_2\otimes\check\beta_1        
        \mod
        \calK_\eff,
      \end{align*}
      in other words
      \begin{math}
        (h_2-h_1)\otimes\check\beta_1
      \end{math}
      is mapped to zero by $\calK\to{^\ab\calK}$. We will
      distinguish between a number of cases according to the types of
      $R$. In each case we will show
      that
      $(h_1,\beta_1)$ and
      $(h_2,\beta_2)$ are
      in the same $\calA$-orbit according to Proposition~\ref{prop:rootOrbit}.
      This entails $t_1^\ab=t_2^\ab$. Set $g=h_2-h_1$. 
 
      {\bf Case $A_1$:}~~ In this case we have $g\in 2G$ by
      Remark~\ref{rem:aAb}. 

      For the next three cases keep in mind that according to
      the twist decomposition, $g\in k_\Delta G_1+G_2$ if $\beta_1$ is long. I
      those cases $k_\Delta=2$.

      {\bf Case $B_2$:}~~ If $\beta\in\Delta_\sh$ then $g\in 2G_1+G_2$
      by Remark~\ref{rem:aAb}. If $\beta\in\Delta_\lng$ then 
      $g\in G_1+2G_2$ by Remark~\ref{rem:aAb} and thus 
      $g\in 2G_1+2G_2$.

      {\bf Case $B_\ell(\ell\ge3)$:}~~ 
      If $\beta\in\Delta_\sh$ then $g\in 2G_1+G_2$
      by Remark~\ref{rem:aAb}. If $\beta\in\Delta_\lng$ then 
      $g\in 2G_1+G_2$.

      {\bf Case $C_\ell(\ell\ge3)$:}~~ 
      If $\beta\in\Delta_\lng$ then 
      $g\in G_1+2G_2$ by Remark~\ref{rem:aAb} and thus 
      $g\in 2G_1+2G_2$.
      
      In the remaining cases there is only a single $\calA$-orbit in $R$, so
      we are done.
    \end{proof}

    Under the hypotheses of Theorem~\ref{thm:ab} we have

    \begin{cor}
      If $\Delta$ is of type 
      $A_\ell(\ell\ge2)$,
      $D_\ell(\ell\ge4)$,
      $E_\ell(8\ge\ell\ge6)$,
      $F_4$, or
      $G_2$, then $\calU\to\calW$ is an isomorphism.
    \end{cor}

    \begin{proof}
      In the cases above the symmetric system $T^\ab$ is a singleton
      due to Proposition~\ref{prop:rootOrbit}. Since $\calA^\ab$ is proper, it
      is of order two. We have the $T^\ab$-reflection morphisms
      $\calU^\ab\to\calW^\ab\to\calA^\ab$. This forces
      $\calU^\ab\to\calA^\ab$ 
      to be an isomorphism. We are done by Theorem~\ref{thm:ab}
    \end{proof}

    This result generalizes a result about Weyl groups of EARSs of simply laced
    type, proved for the first time in \cite{krylyuk}. 
    It seems very likely that
    other generalizations of this kind will be obtained by
    investigating the sequence  ${\calU^\ab\to\calW^\ab\to\calA^\ab}$ of
    elementary abelian two-groups. Since these are vector spaces over
    the Galois field with two elements, they are subject to arguments
    from linear Algebra.

    \bibliographystyle{alpha}
    \bibliography{../../references/references}

\newcommand{\etalchar}[1]{$^{#1}$}
\begin{thebibliography}{AAB{\etalchar{+}}97}

\bibitem[AAB{\etalchar{+}}97]{EALA_AMS}
B.~N. Allison, S.~Azam, S.~Berman, Y.~Gao, and A.~Pianzola.
\newblock Extended affine {L}ie algebras and their root systems.
\newblock {\em Memoirs of the American Mathematical Society}, 126(603),
  \btxmarchlong 1997.

\bibitem[AS06]{azamA1}
S.~Azam and V.~Shahsanaei.
\newblock Presentation by conjugation for ${A}_1$-type extended affine {W}eyl
  groups.
\newblock Preprint QA/0607149v1 on www.arxiv.org, July 2006.

\bibitem[AS07]{MR2341017}
S.~Azam and V.~Shahsanaei.
\newblock Simply laced extended affine {W}eyl groups (a finite presentation).
\newblock {\em Publ. Res. Inst. Math. Sci.}, 43(2):403--424, 2007.

\bibitem[Aza99]{EAWG}
S.~Azam.
\newblock Extended affine {W}eyl groups.
\newblock {\em J. Algebra}, 214(2):571--624, 1999.

\bibitem[Aza00]{azamReduced}
S.~Azam.
\newblock A presentation for reduced extended affine {W}eyl groups.
\newblock {\em Comm. Algebra}, 28(1):465--488, 2000.

\bibitem[Bou68]{bou}
N.~Bourbaki.
\newblock {\em Groupes et alg\`ebres de {L}ie, Chapitres 4, 5 et 6}.
\newblock Hermann, Paris, 1968.

\bibitem[Bou98]{bou_algebra}
N.~Bourbaki.
\newblock {\em Commutative algebra. {C}hapters 1--7}.
\newblock Springer-Verlag, Berlin, 1998.
\newblock Translated from the French, Reprint of the 1989 English translation.

\bibitem[HKT90]{torresani}
R.~H{\o}egh-Krohn and B.~Torr{\'e}sani.
\newblock Classification and construction of quasisimple {L}ie algebras.
\newblock {\em J. Funct. Anal.}, 89(1):106--136, 1990.

\bibitem[Hof07]{myWeyl}
G.~W. Hofmann.
\newblock {W}eyl groups with {C}oxeter presentation and presentation by
  conjugation.
\newblock {\em J. Lie Theory}, 17:337--355, 2007.

\bibitem[Hum92]{hum}
J.~E. Humphreys.
\newblock {\em Reflection Groups and {C}oxeter Groups}.
\newblock Cambridge University Press, 1992.

\bibitem[Kry00]{krylyuk}
Y.~Krylyuk.
\newblock On automorphisms and isomorphisms of quasi-simple {L}ie algebras.
\newblock {\em J. Math. Sci. (New York)}, 100(1):1944--2002, 2000.

\bibitem[Loo69]{MR0239005}
O.~Loos.
\newblock {\em Symmetric spaces. {I}: {G}eneral theory}.
\newblock W. A. Benjamin, Inc., New York-Amsterdam, 1969.

\bibitem[Sai85]{saito1}
K.~Saito.
\newblock Extended affine root systems. {I}. {C}oxeter transformations.
\newblock {\em Publ. Res. Inst. Math. Sci.}, 21(1):75--179, 1985.

\bibitem[Yos04]{yoshii-rootabelian}
Y.~Yoshii.
\newblock Root systems extended by an abelian group and their {L}ie algebras.
\newblock {\em J. Lie Theory}, 14(2):371--394, 2004.

\end{thebibliography}
  \end{document}